\documentclass[12pt]{amsart}

\usepackage[margin=1in]{geometry}
\usepackage{enumerate}
\usepackage{ifpdf}
\usepackage{amsmath}
\usepackage{amsfonts}
\usepackage{amssymb}
\usepackage{amsthm}
\usepackage{amscd}
\usepackage{mathtools}
\usepackage[hyperfootnotes=false]{hyperref}
\usepackage{setspace}
\usepackage{amsrefs}
\usepackage{graphicx}
\usepackage[dvipsnames]{xcolor}

\usepackage{pgf,tikz}
\usetikzlibrary{arrows}  
\usepackage{pgfplots}

\usepackage{rotating}
    \usetikzlibrary{decorations.markings}

\newcommand{\ignore}[1]{}

\newcommand{\db}{{\overline{\partial}}}

\renewcommand{\Re}{\operatorname{Re}}
\renewcommand{\Im}{\operatorname{Im}}

\newcommand{\norm}[1]{\left\lVert {#1} \right\rVert}

\newcommand{\C}{{\mathbb{C}}}

\newcommand{\bC}{{\mathbb{C}}}

\newcommand{\sC}{{\mathcal{C}}}

\def\di{\partial}

\def\e{\epsilon}
\def\zb{\overline{z}}

\newcommand{\bea}{\begin{eqnarray*}}
\newcommand{\eea}{\end{eqnarray*}}

\newtheorem{thm}{Theorem}[section]
\newtheorem{mainthm}{Main Theorem}
\newtheorem*{thmnonum}{Theorem}

\newtheorem{lemma}[thm]{Lemma}

\theoremstyle{definition}
\newtheorem{defn}[thm]{Definition}
\newtheorem{example}[thm]{Example}

\theoremstyle{remark}

\author{Dusty Grundmeier}
\address{Department of Mathematics,
Harvard University, Cambridge, MA 02138, USA}
\email{deg@math.harvard.edu}

\author{Lars Simon}
\address{Department of Mathematical Sciences, Norwegian University of Science and Technology, 
Trondheim, NO-7491, Norway}
\email{lars.simon@ntnu.no}

\author{Berit Stens{\o}nes}
\address{Department of Mathematical Sciences, Norwegian University of Science and Technology, 
Trondheim, NO-7491, Norway}
\email{berit.stensones@math.ntnu.no}

\thanks{The third author is supported by the Research Council of Norway, Grant number 240569/F20.}
\title[Sup-norm Estimates for $\overline{\partial}$]
{Sup-norm Estimates for $\overline{\partial}$}

\begin{document}


\maketitle
\thispagestyle{empty}
\enlargethispage{\baselineskip}

%%%%%%%%%%%%%%%%%%%%%%%%%%%%%%%%%%%%%%%%%%%%%%%%%%%%%%%%%%%%%%%%%%%% Intro

\begin{abstract}
We develop a method for proving sup-norm and H{\"o}lder estimates for $\db$ on wide class of finite type pseudoconvex domains in $\bC^n$. A fundamental obstruction to proving sup-norm estimates is the possibility of singular complex curves with exceptionally high order of contact with the boundary. Our method handles this problem, and in $\mathbb{C}^3$, we prove sup-norm and H{\"o}lder estimates for all bounded, pseudoconvex domains with real-analytic boundary.
\end{abstract}

 \section{Introduction} \label{section:intro}
 
A fundamental problem in complex analysis is to estimate solutions to the Cauchy-Riemann equations. In this paper, we study pseudoconvex domains with real-analytic boundary in $\mathbb{C}^n$ where $n\geq 3$. We develop a method to solve the $\db$-equation for $(0,1)$-forms with sup-norm and H{\"o}lder estimates.
 
In the 1970's, Henkin \cite{H1:dbar} and Ramirez \cite{Ram} developed integral kernel techniques for solving $\db$ and proved sup-norm estimates for $\db$ for bounded strictly pseudoconvex domains. However, for general pseudoconvex domains, the situation is more subtle. In \cite{Sib}, Sibony gave an example of a smooth, bounded pseudoconvex domain in $\mathbb{C}^3$ where sup-norm estimates are not possible. Sibony's example shows that pseudoconvexity alone is not enough to solve $\db$ with sup-norm estimates. Therefore it makes sense to restrict to the study to D'Angelo finite type domains (see \cites{JPD:book,JPD:type} for discussions of D'Angelo finite type). In 1986, Forn{\ae}ss \cite{F1:supnorm} proved sup-norm estimates for a wide class of domains in $\bC^2$, including the Kohn-Nirenberg example. In 1988, Fefferman and Kohn \cite{FK:holderestimates} solved the problem for finite type domains in $\C^2$. Finally in 1990, Range \cite{R1:holder} proved H{\" o}lder estimates for pseudoconvex domains of finite type in $\C^2$ using integral kernel methods. In a series of papers, Forn{\ae}ss-Diederich-Wiegerinck \cite{DFW} and Forn{\ae}ss-Diederich-Fischer \cite{DFF} proved sup-norm estimates for $\db$ on convex, finite type domains in higher dimensions using ideas from McNeal \cite{McN:convex}. However, there has been little progress in the last thirty years and new ideas have been required. 

The geometry is much more complicated in dimensions three or higher. A major difficulty is the possibility of singular complex curves with exceptionally high order of contact with the boundary, and hence the type might change in different directions in subtle ways. Even for relatively simple domains in $\mathbb{C}^3$, the existence of sup-norm estimates is unknown. For instance, it's been an open problem to prove sup-norm estimates for domains where the type is the same in all directions. In this paper we deal with these added complexities in higher dimensions. For domains where the type is the same in all directions, our method establishes sup-norm estimates in dimension 3 or higher.  Furthermore, we prove sup-norm estimates even in cases where there are curves with exceptionally high order of contact with the boundary, and in $\mathbb{C}^3$, we completely handle these added difficulties and establish sup-norm estimates for bounded pseudoconvex domains with real-analytic boundary. More precisely, we prove the following two main theorems.

 \begin{mainthm} \label{thm:supnorm}  Suppose $\Omega$ is a bounded pseudoconvex domain with real-analytic boundary such that if $p\in \di \Omega$ then locally there exists another pseudoconvex domain $\Omega_p^{*}$ and a function $\Phi$ such that 
 \begin{enumerate}
 \item $\overline{\Omega}\setminus \{ p\} \subset \Omega^{*}_p$,
 \item $|\Phi(q)| \sim \text{dist}(q, \di \Omega_{p}^{*})$ when $q \in \di \Omega$, and
 \item $\{\Phi =0\} \cap \Omega_p^{*}=\{p\}$.
 \end{enumerate}
 If $f$ is a $\overline{\partial}$-closed $(0,1)$-form on $\overline{\Omega}$, then there exists a solution $u$ of $\overline{\partial} u = f$ on $\Omega$ such that
 \begin{equation*}
 ||u||_{\infty} \leq C_\Omega ||f||_{\infty}
 \end{equation*}
 where $C_\Omega$ is independent of $f$. 
 
 In fact, if $\Omega$ has D'Angelo type $2L$, then for every $\delta >0$, there is a solution $u=u_\delta$ as above that satisfies $\left(\frac{1}{2L}-\delta\right)$-H{\"o}lder estimates with constant depending only on $\Omega$ and $\delta$. 
 \end{mainthm}
 
 Thus Theorem \ref{thm:supnorm} reduces the sup-norm estimates problem to constructing the bumped-out domain $\Omega_p$ and the function $\Phi$. This reduction gives many new situations where we can give sup-norm estimates in $\mathbb{C}^n$ for $n\geq 3$. Using the results of Noell \cite{Noell}, Bharali and Stens{\o}nes \cite{BS:plurisubharmonic}, Bharali \cite{B:model}, and Forn{\ae}ss and Stens{\o}nes \cite{FS:alg}, and Simon \cite{Lars}, one can see wide classes of domains in $\bC^n$ where we can construct $\Omega_p^{*}$ and $\Phi$. We provide examples in the next section. 
 
 In $\mathbb{C}^3$, we explicitly construct the bumped-out domain $\Omega_p^{*}$ and the function $\Phi$, and hence we completely solve the problem of finding sup-norm estimates for $\db$ in $\mathbb{C}^3$.
 
  \begin{mainthm} \label{thm:supnormC3}  Suppose $\Omega$ is a bounded pseudoconvex domain with real-analytic boundary of finite D'Angelo type $2L$ in $\bC^3$ and $f$ is a $\overline{\partial}$-closed $(0,1)$-form on $\overline{\Omega}$.  Then there exists a solution $u$ of $\overline{\partial} u = f$ on $\Omega$ such that
 \begin{equation*}
 ||u||_{\infty} \leq C_\Omega ||f||_{\infty}
 \end{equation*}
 where $C_\Omega$ is independent of $f$. Furthermore, for every $\delta >0$, there is a solution $u=u_\delta$ as above that satisfies $\left(\frac{1}{2L}-\delta\right)$-H{\"o}lder estimates with constant depending only on $\Omega$ and $\delta$.
 \end{mainthm}

 We conclude the introduction with an outline of the rest of the paper. We focus on the proofs of the main theorems in $\mathbb{C}^3$. In section \ref{section:examples}, we give examples of classes of domains where our method applies. The main approach is to solve the Cauchy-Fantappie equation pointwise. We then create an integral kernel on a smaller domain $\Omega_\e$ and get uniform estimates on the smaller domain. We then use a normal families argument to get estimates on $\Omega$. Section \ref{section:ingredients} develops this machinery and shows how we use the Cauchy-Fantappie equation. In section \ref{section:koszul}, we show how to use a Koszul complex to modify our smooth solutions to the Cauchy-Fantappie equation. In section \ref{section:phi} and \ref{section:ps}, we show how to build a non-holomorphic support function. In section \ref{section:weights}, we develop the plurisubharmonic weights needed to use the full extent of H{\"o}rmander's $L^2$ theory. In section \ref{section:omegaestimate}, we give $L^2$-estimates, and in section \ref{section:01estimates}, we get pointwise estimates. We prove the main theorem in section \ref{section:proof}. Finally, in section \ref{sectionhigherdim}, we show how to reduce the problem to a bumping problem in the higher-dimensional case.

 %%%%%%%%%%%%%%%%%%%%%%%%%%%%%%%%%%%%%%
 \section{Examples} \label{section:examples}
 
 In this section, we give examples of wide classes of domains where our method applies.

 \begin{example} (A. Noell \cite{Noell}). Suppose $P(z_1, \dots, z_{n-1})$ is a homogeneous, plurisubharmonic polynomial of degree $2k$ on $\mathbb{C}^{n-1}$ that is not harmonic along any complex line through the origin. 
 Let 
 \begin{equation*}
 \Omega = \{ (\zeta, z_1, \dots, z_{n-1}): \Re(\zeta) + P(z_1,\dots, z_{n-1})<0\}.
 \end{equation*}
Using a result of Noell \cite{Noell}, there exists the bumped out domain $\Omega_p^{*}$ and we use $\Phi=\zeta-A\norm{(z_1,\dots,z_{n-1})}^{2k}$, and hence Theorem \ref{thm:supnorm} applies.
 \end{example} 
 
 \begin{example}
 Let 
 \begin{equation*}
 \Omega = \{ (\zeta, z_1, \dots, z_{n-1}): \Re(\zeta) + \sum_{j=1}^{k} |f_j(z_1, \dots, z_{n-1})|^2<0\} 
 \end{equation*}
 where the common zero set of $f_1$, \dots $f_k$ is 0 and $\Omega$ is finite type. Then we take $$\Omega_0^{*} =\{\Re(\zeta) +(1-\e) \sum_{j=1}^k{|f_j(z_1,\dots, z_{n-1})|^2} <0\}$$
 and $$\Phi = \zeta - A \sum_{j=1}^{k} {|f_j(z_1, \dots, z_{n-1})|^2}.$$ Nearby boundary points are of the same kind.
 \end{example} 
 
  \begin{example}
 Let 
 \begin{equation*}
 \Omega = \{ z \in \mathbb{C}^n: \sum_{j=1}^{k} |f_j(z)|^2<1\}.
 \end{equation*}
Here we can locally transform the domain to be of the form in the previous example.
 \end{example} 
 
 \begin{example}
 Let 
 \begin{equation*}
 \Omega = \{ (\zeta, z_1, z_2, z_3, z_4 ): \Re(z_1) + |z_3^2-z_4^3|^6|z_2|^2+|z_3^2-z_4^3|^8+\frac{15}{7}|z_3^2-z_4^3|^2\Re(z_3^2-z_4^3)^6+\norm{z}^{10}<0\}.
 \end{equation*}
 \end{example}

 %%%%%%%%%%%%%%%%%%%%%%%%%%%%%%%%%%%%%%
 \section{Key Ingredients in the Theorem} \label{section:ingredients}
 
  We briefly highlight some key techniques in the paper now. The fundamental approach of this paper is to use integral kernel techniques. We use an idea inspired by a comment of Range given in a workshop in Beijing to construct a non-holomorphic support function and then solve a smooth division problem. Using a Koszul complex and the full extent of H{\"o}rmander $L^2$-techniques, we modify these functions to obtain holomorphic solutions to the Cauchy-Fantappie equations. Finally we use ``pseudoballs" (see Catlin \cite{C1:invariantmetrics} and McNeal \cite{McN:convex}) and subaveraging to pass from $L^2$ estimates to pointwise estimates.
 
 \subsection{Bumping to Type}
 
 We begin by giving a precise definition for bumping.
 
 \begin{defn}
 Given a pseudoconvex domain $\Omega$ and $p\in \partial \Omega$, then $\Omega$ can be {\it locally bumped at} $p$ if there exists a neighborhood $U$ of $p$ and a larger pseudoconvex domain $\Omega_p^{*}$ such that $\overline{\Omega}  \setminus \{p\} \cap U \subset \Omega_p^{*}$. We then say that $\Omega_p^{*}$ is a {\it local bumping at} $p$.
 \end{defn}
 
 If $\Omega \subset \C^2$ and $p$ is of type $2k$, then $\Omega_p^{*}$ can be chosen so that boundaries meet to order $2k$ in the complex tangential direction. In $\bC^n$ for $n\geq 3$, there are added difficulties from the additional complex tangential directions. For example, type will change in different complex directions. Even more, there might be singular complex curves with maximal order of tangency.
 
 In \cite{DF}, Diederich and Forn{\ae}ss show that if $\Omega$ is pseudoconvex and of finite type at $p\in\partial\Omega$, then $\Omega$ can be bumped to some high order at $p$ (potentially much higher than the type). For our construction we need to bump to the lowest possible order in all directions. In order to make this precise we first need to define what it means for a polynomial to be bumpable.
 
 \begin{defn}\label{defbumppoly}
Let $P$ be a homogeneous plurisubharmonic polynomial on $\mathbb{C}^{n-1}$. We say that $P$ {\it can be bumped} if there exists a plurisubharmonic function $H$, smooth away from $0$ and homogeneous of the same degree as $P$, such that for some small $\epsilon >0$ we have $H\leq P-\epsilon |P|$ with equality precisely in $0$ and along the complex lines through $0$ along which $P$ is harmonic.

In the {\it weighted}-homogeneous case this is defined by homogenizing in the obvious way.
 \end{defn}

 \begin{defn}
 Let $\Omega\subseteq\mathbb{C}^n$ be a pseudoconvex domain. We take $z=(\zeta, z')\in \bC \times \bC^{n-1}$. We say that $\Omega$ is {\it bumpable to type at} $p\in\partial\Omega$ if it is locally contained in a pseudoconvex domain $\widetilde{\Omega}$ with $p\in\partial\widetilde{\Omega}$, which locally at $p$ is given as
\begin{align*}
    \{\Re(\zeta ) +\sum_{j=1}^{J}M_j(z') +\norm{z'}^{2M}+\mathcal{O}(|\xi |^2,|\Im (\xi )|\norm{z'})<0 \},
\end{align*} 
where each of the $M_j$ is a weighted-homogeneous plurisubharmonic polynomial that can be bumped as in Definition \ref{defbumppoly}.

If $\Omega$ is bumpable to type at all of its boundary points, we simply say that $\Omega$ is {\it bumpable to type}.
 \end{defn}

\subsection{Henkin Integral Kernel}

We use the Henkin integral kernel to solve $\db$ and obtain sup-norm estimates. Let $f= \sum_{i}{f_i d \zb_i}$ be a closed $(0,1)$-form and 
\begin{equation*}
S_\Omega f = c_n \int_{\di \Omega \times [0,1]}f \wedge \eta(w) \wedge \omega(\zeta)
-c_n \int_\Omega \frac{f(\zeta)}{\norm{\zeta-z}^{2n}} \eta(\overline{\zeta} -\zb) \wedge \omega(\zeta)
\end{equation*}
where 
\begin{align*}
w(\zeta) & =d\zeta_1 \wedge \dots \wedge d\zeta_n \text{ and }\\ 
\eta(\zeta) & = \sum_{i=1}^n (-1)^{i-1}\zeta_id\zeta_1\wedge \cdots \wedge \widehat{d\zeta_i}\wedge
\cdots \wedge d\zeta_n.
\end{align*}
Further let 
\begin{align*}
w(\zeta) &= (w_1, \dots, w_n) \text{ and }\\
w_i &= \lambda \frac{\zeta_i-z_i}{\norm{\zeta-z}^2}+(1-\lambda) h_i(\zeta, z)
\end{align*}
where $\lambda \in [0,1]$ and $h_1, \dots, h_n$ solves the Cauchy-Fantappie equation
\begin{equation*}
\sum_{i=1}^{n}{h_i(\zeta,z)(\zeta_i-z_i) \equiv 1}
\end{equation*}
when $\zeta  \in \partial \Omega$, $z\in \Omega$, and $z \mapsto h_i(\zeta,z)$ is holomorphic in $\Omega$. Then $\db S_\Omega(f) =f$. Our goal is to construct the functions $h_i$ such that $$\norm{S_\Omega f}_\infty \leq C_\Omega \norm{f}_\infty.$$

The challenge is to show that if $f$ is bounded by a constant $C'$, then there is a constant $C$ which only depends on $C'$ and $\Omega$ such that 
\begin{equation*}
\left| \int_{\di \Omega \times [0,1]} f \wedge \eta(w) \wedge \omega(\zeta) \right| < C.
\end{equation*}

When $\Omega\subseteq\mathbb{C}^3$, then
\begin{equation*}
\eta(w) = w_1 dw_2 \wedge dw_3 -w_2 dw_1 \wedge dw_3 + w_3 dw_1 \wedge dw_2,
\end{equation*} and we can expand a typical term as follows
\begin{equation*}
w_i dw_j \wedge dw_k = w_i \sum_{m=1}^3{\left[ \frac{\di w_j}{\di \lambda}\frac{w_k}{\di \overline{\zeta_m}}-\frac{\di w_j}{\di \overline{\zeta_m}}\frac{w_k}{\di \lambda} \right]d\lambda \wedge d\overline{\zeta_m}}+ \sum{u_{m,n} d\overline{\zeta_n}\wedge d\overline{\zeta_m}}
\end{equation*}

The above integral includes the $(0,1)$-form $f$ and the $(3,0)$-form $\omega$ and the real dimension of $\di \Omega$ is five. Thus the integral cannot support the terms $\sum{u_{m,n} d\overline{\zeta_n}\wedge d\overline{\zeta_m}}$. Therefore we need only study the expressions
\begin{equation*}
w_i \sum_{m=1}^3{\left[ \frac{\di w_j}{\di \lambda}\frac{w_k}{\di \overline{\zeta_m}}-\frac{\di w_j}{\di \overline{\zeta_m}}\frac{w_k}{\di \lambda} \right]d\lambda \wedge d\overline{\zeta_m}}.
\end{equation*} When we calculate $\eta(w)$ and ignore the terms that cannot be supported in the integral over $\di \Omega$, we get 
\begin{equation*}
\eta(w) = \lambda^2 B- \lambda \eta_1 -(1-\lambda) \eta_2
\end{equation*}
where $B$ is the Bochner-Martinelli kernel and

\begin{align*}
\eta_1 & =  \sum_{n=1}^3 \Bigg\{
\frac{\overline{\zeta}_1-\overline{z}_1}{\|\zeta-z\|^2}
\Bigg[h_2\left(  \frac{\delta_{n,3}}{\|\zeta-z\|^2}-\frac{(\overline{\zeta}_3-\overline{z}_3 )(\zeta_n-z_n)}{\|\zeta-z\|^4}                                                \right)\\
&  -  h_3\left(  \frac{\delta_{n,2}}{\|\zeta-z\|^2}-\frac{(\overline{\zeta}_2-\overline{z}_2 )(\zeta_n-z_n)}{\|\zeta-z\|^4}                                                \right) \Bigg]\\
& -  \frac{\overline{\zeta}_2-\overline{z}_2}{\|\zeta-z\|^2}
\Bigg[h_1\left(  \frac{\delta_{n,3}}{\|\zeta-z\|^2}-\frac{(\overline{\zeta}_3
	-  \overline{z}_3 )(\zeta_n-z_n)}{\|\zeta-z\|^4}                                                \right)\\
&  -    h_3\left(  \frac{\delta_{n,1}}{\|\zeta-z\|^2}-\frac{(\overline{\zeta}_1-\overline{z}_1 )(\zeta_n-z_n)}{\|\zeta-z\|^4}                                                \right) \Bigg]\\
& +  \frac{\overline{\zeta}_3-\overline{z}_3}{\|\zeta-z\|^2}
\Bigg[h_1\left(  \frac{\delta_{n,2}}{\|\zeta-z\|^2}-\frac{(\overline{\zeta}_2
	-  \overline{z}_2)(\zeta_n-z_n)}{\|\zeta-z\|^4}                                                \right) \\
& -   h_2\left(  \frac{\delta_{n,1}}{\|\zeta-z\|^2}-\frac{(\overline{\zeta}_1-\overline{z}_1 )(\zeta_n-z_n)}{\|\zeta-z\|^4}                                                \right) \Bigg]\Bigg\}d\lambda\wedge d\overline{\zeta}_n
\end{align*}  
and 
\begin{align*}
\eta_2 & =  \sum_{n=1}^3 \Bigg\{
\frac{\overline{\zeta}_1-\overline{z}_1}{\|\zeta-z\|^2}\left(h_2\frac{\partial h_3}{\partial \overline{\zeta}_n}-h_3\frac{\partial h_2}{\partial \overline{\zeta}_n}\right)\\
& -  \frac{\overline{\zeta}_2-\overline{z}_2}{\|\zeta-z\|^2}\left(h_1\frac{\partial h_3}{\partial \overline{\zeta}_n}-h_3\frac{\partial h_1}{\partial \overline{\zeta}_n}\right)\\
& +  \frac{\overline{\zeta}_3-\overline{z}_3}{\|\zeta-z\|^2}\left(h_1\frac{\partial h_2}{\partial \overline{\zeta}_n}-h_2\frac{\partial h_1}{\partial \overline{\zeta}_n}\right) \Bigg\}d\lambda \wedge d\overline{\zeta}_n.
\end{align*}

We see that $\|\eta_1\|$ has singularities of order $|h_i|\frac{1}{\|\zeta-z\|^3}.$

The integral that is the most difficult to estimate is
$$
\int_{\di \Omega \times [0,1]} f\wedge \eta_2\wedge \omega.
$$

Observe that $\omega(\zeta)=d\zeta_1\wedge d\zeta_2 \wedge d\zeta_3$         
already has a differential which is orthogonal to the complex tangential direction, so
$$
\int_{\di \Omega \times [0,1]}f\wedge \eta_2\wedge \omega
$$
can only support differentials $d\overline{\zeta}_n$ from $\eta_i$ that is complex tangential to $\partial \Omega.$

Therefore we only need to estimate the integrals with terms of the form
$$
\frac{\overline{\zeta}_i-\overline{z}_i}{\|\zeta-z\|^2}\left(h_j\frac{\partial h_k}{\partial \overline{\zeta_n}}
-h_k\frac{\partial h_j}{\partial \overline{\zeta_n}}\right).
$$
 
 \subsection{Pointwise Solutions to Cauchy-Fantappie Equation}
 
We are not able to solve the Cauchy-Fantappie equation with solutions that are smooth in the boundary variable. Instead, we solve the Cauchy-Fantappie equation pointwise; i.e.\ given $p=(\eta_1^0,\eta_2^0,\eta_3^0)\in \di \Omega$, we find $h_1$, $h_2$, $h_3$ such that 
\begin{equation*}
\sum_{j=1}^3{h_j(p,z)(\eta_j^0-z_j)\equiv 1}
\end{equation*}
where $h_j$ is holomorphic in $z$. The resulting integral kernel would be nicely integrable if $h_j$ were continuous in $p$; instead, we need to use additional techniques from \cite{R1:holder} to construct a sequence of integral kernels on slightly smaller domains that give uniform estimates. We then use a standard normal families argument to give sup-norm estimates on the original domain. 

We now need to choose good smooth solutions $g_j$ which can be modified using a Koszul complex with H{\"o}rmander's $L^2$-theory. Unfortunately, using the usual smooth solutions to the division problem as in Skoda \cite{S:l2estimates} does not yield sufficient estimates. We therefore need to use a more careful choice of smooth solutions. Our choice is inspired by a suggestion of Range in a lecture in Beijing. This choice is designed to reflect the type at a boundary point in every complex tangential ``direction.'' 

More precisely, we will use the bumping to show that locally there exists $\Phi$ such that:
\begin{enumerate}
    \item $\Phi=(\eta_1^0-z_1)-F((\eta_2^0-z_2),\overline{(\eta_2^0-z_2)},(\eta_3^0-z_3),\overline{(\eta_3^0-z_3)}),$
    \item $F>0$ away from $(0,0),$  
    
    \item $\{ \Phi=0\}\cap \Omega_p^{*}=\emptyset$,
    \item $\left|\Phi\big|_{\overline{\Omega}}\right|\sim  \text{dist}(\cdot,\di \Omega_p^{*})$.
\end{enumerate}

Now we let $g_1=\frac{1}{\Phi}$, $g_2=\frac{P_2}{\Phi}$, and $g_3=\frac{P_3}{\Phi}$ such that 
\begin{equation*}
    \frac{1}{\Phi}(\eta_1^{0}-z_1)+\frac{P_2}{\Phi}(\eta_2^0-z_2)+\frac{P_3}{\Phi}(\eta_3^{0}-z_3)\equiv 1.
\end{equation*}

Finally, we use the following version of H{\"o}rmander's theorem.
 
 \begin{thm} (H\"ormander, Demailly) \label{thm:hor}
Let $\rho$ be a plurisubharmonic function on $D\subset \mathbb C^n$, pseudoconvex, $v$ is a $\overline\partial-$
closed $(0,q)-$form. Then there exists a $(0,q-1)-$form such that $\overline\partial u=v$ and

$$
\int_D |u|^2 e^{-\rho} \leq C \int_D <A^{-1}v, v> e^{-\rho}\text{,}
$$
where $A$ depends on $\rho$ and $q$. In case $q=1$, the matrix $A$ is just the Complex Hessian matrix of $\rho$.
\end{thm}

This result together with the subaveraging principle will give the desired estimates.
 
 %%%%%%%%%%%%%%%%%%%%%%%%%%%%%%%%%%%%% Koszul complex

\section{Koszul Complex} \label{section:koszul}

Given the choice of smooth solutions $g_j$ from the last section, we now illustrate how we modify them to get holomorphic solutions. While this technique is well-known and standard, we develop the expressions explicitly in order to see exactly what kind of estimates we obtain.

In order to simplify notation we write $\eta$ for $\eta^0$. We start with smooth $g_1,g_2,g_3$ in $\Omega^*_\eta$
such that

$$
g_1({\eta},z)(\eta_1-z_1)+g_2({\eta},z)(\eta_2-z_2)+g_3({\eta},z)(\eta_3-z_3) \equiv 1.
$$

This gives

$$
\overline\partial g_1({\eta},z)(\eta_1-z_1)+\overline\partial g_2({\eta},z)(\eta_2-z_2)+\overline\partial g_3({\eta},z)(\eta_3-z_3) =0,
$$
and hence
\bea
\overline\partial g_1 & = & \overline\partial g_1({\eta},z)(g_1({\eta},z)(\eta_1-z_1)+g_2({\eta},z)(\eta_2-z_2)+g_3({\eta},z)(\eta_3-z_3))\\
& = & g_1\overline\partial g_1 (\eta_1-z_1)+g_2\overline\partial g_1 (\eta_2-z_2)+g_3\overline\partial g_1 (\eta_3-z_3)\\
& = & g_1[ -   \overline\partial g_2({\eta},z)(\eta_2-z_2)-\overline\partial g_3({\eta},z)(\eta_3-z_3)          ]+g_2\overline\partial g_1 (\eta_2-z_2)+g_3\overline\partial g_1 (\eta_3-z_3)\\
& = & [g_2\overline\partial g_1-g_1\overline\partial g_2](\eta_2-z_2)+[g_3\overline\partial g_1-g_1\overline\partial g_3](\eta_3-z_3).
\eea
Similarly we obtain
$$
\overline\partial g_2  =-[g_2\overline\partial g_1-g_1\overline\partial g_2](\eta_1-z_1)+[g_3\overline\partial g_2-g_2\overline\partial g_3](\eta_3-z_3),
$$
and
$$
\overline\partial g_3  =[g_3\overline\partial g_1-g_1\overline\partial g_3](\eta_1-z_1)+[g_3\overline\partial g_2-g_2\overline\partial g_3](\eta_2-z_2).
$$

Simplifying notation, we introduce the following
\bea
\overline{\partial} g_1 & = & h_{1,2}(\eta_2-z_2)+h_{1,3}(\eta_3-z_3)\\
\overline{\partial} g_2 & = & -h_{1,2}(\eta_1-z_1)+h_{2,3}(\eta_3-z_3)\\
\overline{\partial} g_3 & = & -h_{1,3}(\eta_1-z_1)-h_{2,3}(\eta_2-z_2)
\eea
where
$$
h_{i,j} = g_j\overline\partial g_i-g_i\overline\partial g_j.
$$

Thus
\bea
\overline{\partial} h_{1,2}(\eta_2-z_2)+\overline\partial h_{1,3}(\eta_3-z_3)  & = & 0\\
-\overline{\partial} h_{1,2}(\eta_1-z_1)+\overline\partial h_{2,3}(\eta_3-z_3)  & = & 0\\
-\overline{\partial} h_{1,3}(\eta_1-z_1)+\overline\partial h_{2,3}(\eta_2-z_2)  & = & 0.
\eea
Hence
\bea
\overline\partial h_{1,2} & = & \overline\partial h_{1,2}(g_1({\eta},z)(\eta_1-z_1)+g_2({\eta},z)(\eta_2-z_2)+g_3({\eta},z)(\eta_3-z_3))\\
& = & g_1 \overline\partial h_{1,2}(\eta_1-z_1)+g_2 \overline\partial h_{1,2}(\eta_2-z_2)+g_3 \overline\partial h_{1,2}(\eta_3-z_3)\\
& = & g_1 \overline\partial h_{2,3}(\eta_3-z_3)-g_2 \overline\partial h_{1,3}(\eta_3-z_3)+g_3 \overline\partial h_{1,2}(\eta_3-z_3)\\
& = & [g_1 \overline\partial h_{2,3}-g_2 \overline\partial h_{1,3}+g_3 \overline\partial h_{1,2}](\eta_3-z_3)\\
\overline\partial h_{1,3} & = &  -[g_1 \overline\partial h_{2,3}-g_2 \overline\partial h_{1,3}+g_3 \overline\partial h_{1,2}](\eta_2-z_2)\\
\overline\partial h_{2,3} & = &  [g_1 \overline\partial h_{2,3}-g_2 \overline\partial h_{1,3}+g_3 \overline\partial h_{1,2}](\eta_1-z_1),
\eea
or
\bea
\overline\partial h_{1,2} & = & \omega (\eta_3-z_3)\\
\overline\partial h_{1,3} & = & -\omega (\eta_2-z_2)\\
\overline\partial h_{2,3} & = & \omega (\eta_1-z_1)
\eea
where 
$$
\omega=g_1 \overline\partial h_{2,3}-g_2 \overline\partial h_{1,3}+g_3 \overline\partial h_{1,2}.
$$
From the above we also see that $\omega$ is a closed $(0,2)$-form. We will use H\"ormander's Theorem (Theorem \ref{thm:hor} above) to solve $\overline\partial$ for the occurring $(0,2)$-form and $(0,1)$-forms with an individual weight for each of the forms.

We will make a careful choice of a plurisubharmonic weight $\psi_0$ such that $\omega\in L^2(\psi_0)$,
and find a $(0,1)$-form $u\in L^2(\psi_0)$ such that
$$
\overline\partial u = \omega.$$

Then
\begin{itemize}
\item $h_{1,2}-(\eta_3-z_3)u$
\item $h_{1,3}+(\eta_2-z_2)u$
\item $h_{2,3}-(\eta_1-z_1)u$
\end{itemize}

are all closed forms and
\bea
\overline\partial g_1 & = & (h_{1,2}-(\eta_3-z_3)u)(\eta_2-z_2)-(h_{1,3}+(\eta_2-z_2)u)(\eta_3-z_3)\\
\overline\partial g_2 & = & -(h_{1,2}-(\eta_3-z_3)u)(\eta_1-z_1)+(h_{2,3}-(\eta_1-z_1)u)(\eta_3-z_3)\\
\overline\partial g_3 & = & -(h_{1,3}+(\eta_2-z_2)u)(\eta_1-z_1)-(h_{2,3}-(\eta_1-z_1)u)(\eta_2-z_2).
\eea

Next we need to find good, minimal, weights $\psi_1,\psi_2$ and $\psi_3$ such that if $u\in L^2(\psi_0),$ then

\begin{itemize}
\item $h_{1,2}-(\eta_3-z_3)u$ is in $L^2(\psi_1)$
\item $h_{1,3}+(\eta_2-z_2)u$ is in $L^2(\psi_2)$
\item $h_{2,3}-(\eta_1-z_1)u$ is in $L^2(\psi_3)$
\end{itemize}

Then we find $v_1,v_2$ and $v_3$, functions in $L^2(\psi_1),L^2(\psi_2)$ and $v_3\in L^2(\psi_3)$
such that
\bea
\overline{\partial} v_1 & = & h_{1,2}-(\eta_3-z_3)u\\
\overline{\partial} v_2 & = & h_{1,3}+(\eta_2-z_2)u\\
\overline{\partial} v_3 & = & h_{2,3}-(\eta_1-z_1)u.
\eea

We now let

\bea
h_1 & = & g_1-v_1(\eta_2-z_2)-v_2(\eta_3-z_3)\\
h_2 & = & g_2+v_1(\eta_1-z_1)-v_3(\eta_3-z_3)\\
h_3 & = & g_3+v_2(\eta_1-z_1)+v_3(\eta_2-z_2).
\eea

Recall $g_j=\frac{P_j}{\Phi}$ where $P_1=1$. When we go through the calculations we see that
$$
\omega=2\frac{\overline\partial P_3\wedge \overline\partial P_2}{\Phi^3}.
$$

Now we need to choose $\Phi,P_1,P_2$ and $P_3$.

%%%%%%%%%%%%%%%%%%%%%%%%%%%%%%%%%%%%%%

\section{Setting up $\Phi$} \label{section:phi}

The next part of the paper will deal with the choice of the functions $g_1,g_2$ and $g_3.$  The critical part is to carry out the construction locally near $\eta$. Afterwards one simply extends them to $\Omega_{\eta}^*$. Now, for ease of notation, we assume that ${\eta}=0$. Locally around $0$, the domain $\Omega$ is given as
\begin{align*}
\{\operatorname{Re}({\xi})+r(z,w)+s({\xi},z,w)<0\}\text{,}
\end{align*}
where $s$ and $r$ are real-analytic, $s({\xi},z,w)=\mathcal{O}({|{\xi}|}^2,\Vert{(z,w)}\Vert\cdot{}|\operatorname{Im}({\xi})|)$ and $r$ does not have any pluriharmonic terms.

We want to first choose a support function:

$$
\Phi{({\xi},z,w)}={\xi}-F(z,w).
$$

This function is not holomorphic as a function of $z$ and $w$ but will be chosen related to how the bumped
domain $\Omega^*_0$ looks.

Finally we will concretely solve a division problem such that

$$
\frac{P_1}{\Phi}\xi +\frac{P_2}{\Phi}z+\frac{P_3}{\Phi}w\equiv 1
$$
in $\Omega^*_0.$

By Diederich-Forn{\ae}ss \cite{DF} there exist a large $M>0$ and a real-valued polynomial $R(z,w)$, without pluriharmonic terms, such that the following domain is pseudoconvex and locally contains $\Omega$:

\begin{align*}
\widetilde{\Omega}=  \{\operatorname{Re}({\xi})+R(z,w)+C\Vert{(z,w)}\Vert^{2M}+s({\xi},z,w)<0\}\text{.}
\end{align*}
Since the construction of $\Phi$ only depends on the complex tangency of complex curves to the boundary of $\Omega$ at $0$, we will consider $\widetilde{\Omega}$ instead of $\Omega$ and $R$ instead of $r$ for the remainder of this section.

\subsection{Initial Examples} 

The simplest case is if the lowest order term $H_{2k}$ in $R$ is not harmonic along any complex lines through the origin. Then Noell \cite{Noell} showed the domain can be bumped to order $2k$ in all complex tangential directions. In this case we choose $F=A|z|^{2k}+A|w|^{2k}$ where $A$ is a large positive constant and $P_2=-Az^{k-1}\overline{z}^k$ and $P_3=-Aw^{k-1}\overline{w}^{k}$.

From [BS] we know that $H_{2k}$ can only be harmonic along finitely many complex lines through $0$; we denote these lines by
$L_1,L_2,\dots,L_m$. For simplicity let us assume that none of them is the $z-$ axis, so each line is of the form
$L_i=\{(z,w); z=\tau_i w\}.$

The next simplest case is if $R-H_{2k}$ is plurisubharmonic. In this case, near each line $R-H_{2k}=q_{2K_i}(w)+{\mbox{ higher order terms in }}w+\mathcal O((z-\tau_i w)w).$

Further, by changing holomorphic coordinates if need be we may assume that ${H_{2k}}_{I L_i}\equiv 0.$

Near a given $L_i$ we can write 
$$
H_{2k}= Q_{2j_i, 2k-2j_i}((z-\tau_i w),\overline{(z-\tau_i w)},w,\overline{w})+
{\mbox{ terms of order larger than }} 2j_i {\mbox{ in }} (z-\tau_i w).
$$

Here $Q_{2j_i,2k-2j_i} $ is homogeneous in $z-\tau_iw$ and $w$ separately.

Since $H_{2k}$ is plurisubharmonic, it follows that also $Q_{2j_i,2k-2j_i}$ is plurisubharmonic ([BS]).
Moreover $Q_{2j_i,2k-2j_i}=s(z-\tau_iw)^\gamma w^\beta$ where $s$ is subharmonic.

From [BS] it follows that in most cases there exist a function 
$$
B(z,\overline{z},w,\overline{w})\geq |z-\tau_i w|^{2k}+|z-\tau_i w|^{2j_i}|w|^{2k-2j_i}
$$

such that we can find a plurisubharmonic $\tilde{H}_{2k}=\tilde{Q}_{2j_i,2k-2j_i}+R$ and
$H_{2k}\geq \tilde{H}_{2k}+\epsilon B$ for some $\epsilon>0.$

Choose a large $A>0$ and near $L_i$ we let
$$
\Phi_i=\xi-A|z-\tau_i w|^{2k}-A|z-\tau_i w|^{2j_i}|w|^{2k-2j_i}-A|w|^{2K_i}.
$$

Further we choose
\bea
P_1 & = & 1\\
P^i_2 & = & -A(z-\tau_i w)^{k-1}(\overline{z}-\overline{\tau}_i \overline{w})^k\\
P^i_3 & = & -A(-\tau_i (z-\tau_i w)^{k-1}(\overline{z}-\overline{\tau}_i w)^k
+|z-\tau_i w|^{2j_i}w^{k-j_i-1}\overline{w}^{k-j_i}+w^{K_i-1}\overline{w}^{K_i}).
\eea

Then 
$$ 
\frac{P_1}{\Phi_i}\xi+\frac{P^i_2}{\Phi_i}z+\frac{P^i_3}{\Phi_i}w \equiv 1.
$$

Away from the lines $L_1,\dots, L_m$ we need to glue these choices together. First we choose a partition of unity $\{\chi_i\}_{i= 1}^m $ such that each $\chi_i$ is constant in a conical neighborhood of each line $L_1,\dots,L_m$. Then we let
$$
\Phi= \xi-A\sum \chi_i |z-\tau_i w|^{2k}
-A\sum \chi_i |z-\tau_i w|^{2j_i}w^{2k-2j_i}
-A \sum \chi_i |w|^{2K_i}.
$$
Then we let
\bea
P_1 & = & 1\\
P_2 & = & -A\sum \chi_i(z-\tau_i w)^{k-1}(\overline{z}-\overline{\tau}_i \overline{w})^k\\
P_3 & = & \sum \chi_i P^3_i.
\eea

Finally we see that 
$$ 
\frac{P_1}{\Phi}\xi+\frac{P_2}{\Phi}z+\frac{P_3}{\Phi}w \equiv 1.
$$ 

Examples 1.2 and 1.4 from the introduction are covered by this case. As we can see from the other example domains in the introduction, we also need to deal with curves of higher order of contact, not just lines. In this case the $|w|^{2K_i}$'s need to be replaced by something much more complicated.

\subsection{Idea and first steps}

The main tool for handling exceptional curves is an algorithm developed by Forn{\ae}ss and Stens{\o}nes in \cite{FS:alg}. Each step in their algorithm will contribute terms to the function $\Phi$. As such, in contrast to Forn{\ae}ss and Stens{\o}nes, we have to keep track of every iteration step in the algorithm, which is why we choose to use the language of graph theory to describe the construction of $\Phi$.

We briefly recall the Forn{\ae}ss-Stens{\o}nes algorithm from \cite{FS:alg}. The algorithm is a three step process. We start with an essentially plurisubharmonic polynomial $r(z,w)$ without pluriharmonic terms. First, we find a complex line on which the lowest order terms vanish. Second, we change coordinates to move this line to an axis. Third we use the Newton diagram to find a weighted homogeneous polynomial coming from an extreme edge and find a curve where the lowest order term vanishes. This process repeats until the weighted homogeneous polynomial does not vanish along any curve. More precisely, if the lowest order terms vanish along the line $w_i=\tau_i z_i$, then we do the following change of coordinates $\tilde{z_i}=z_i$ and $\tilde{w_i}=w_i-\tau_i z_i$. Now we write $r$ in the new coordinates. There will be finitely many extreme edges with slope less than negative one. Among those, choose the one with the largest slope. This will give rise to a weighted homogeneous polynomial of degree $(a_i,b_i)$. Now we make the (singular) change of coordinates $(z_{i+1}, w_{i+1}) = (\tilde{z}^{\frac{1}{a_i}},\tilde{w}^{\frac{1}{b_i}})$. See \cite{FS:alg} for more details. If we look at $r$ in the new coordinates, we will get a new lowest order homogeneous polynomial, which will make a contribution to $\Phi$, similar to the ones above, but now in the new coordinates $(z_{i+1},w_{i+1})$. We need help with the book keeping, so we choose the language of graph theory.

%{\color{red}FIXME: Add description of algorithm.}
Before carrying out the construction in details, we give a brief overview over how the graph is obtained from the Forn{\ae}ss-Stens{\o}nes algorithm.
We describe a rooted (undirected) tree $G=(V,E)$, where each node, except for the root, corresponds to a complex line obtained from a sequence of coordinate changes as described in \cite{FS:alg}. Since the algorithm in \cite{FS:alg} terminates after finitely many steps, an initial coordinate change ensures that none of the occurring complex lines is given as $\{v=0\}$ in the complex coordinates $(u,v)$ with respect to which the line is described in the algorithm.

We construct this tree by applying the algorithm from \cite{FS:alg}.

We initialize the tree with its root $(0,0)\in{V}$. If the lowest order homogeneous term of $R$ is not harmonic along any complex line through $0$, we stop. Otherwise, as mentioned previously, that term will be harmonic along only finitely many complex lines through $0$, say $L_{(1,1)},\dots{},L_{(1,l_1)}$, where $l_1\geq{1}$ and $L_{(1,i)}\neq{L_{(1,j)}}$ for $i\neq{j}$. We add nodes $(1,1),\dots{},(1,{l_1})\in{V}$ corresponding to these lines, as well as edges connecting each of these newly introduced nodes to the root, i.e.\ $\{(1,1),(0,0)\},{\dots},\{(1,{l_1}),(0,0)\}\in{E}$.

Now we consider the line $L_{(1,1)}$, which for suitable ${\tau}_{(1,1)}\in\mathbb{C}$ is given as $L_{(1,1)}=\{(z,w)\in\mathbb{C}^2\colon{z-{\tau}_{(1,1)}w=0}\}$. The real-valued polynomial $\widetilde{R}_{(1,1)}$ given by
\begin{align*}
\widetilde{R}_{(1,1)}(\widetilde{z},\widetilde{w})=R(\widetilde{z}+{\tau}_{(1,1)}\widetilde{w},\widetilde{w})
\end{align*}
is harmonic along the complex line $\{(\widetilde{z},\widetilde{w})\in\mathbb{C}^2\colon\widetilde{z}=0\}$. We consider the Newton diagram of $\widetilde{R}_{(1,1)}$. If there exists no extreme edge with slope $<-1$, we stop (if $1=l_1$) or we move on to considering the line $L_{(1,2)}$ (if $1<l_1$). Otherwise let $E_{(1,1)}$ be the extreme edge with the {\emph{largest}} slope among all extreme edges with slope $<-1$ (Caution: this is now an extreme edge in a Newton diagram and {\emph{not}} an edge of the graph). We then find positive integers $k_{(1,1)},l_{(1,1)}$ with $\gcd(k_{(1,1)},l_{(1,1)})=1$, such that the lowest-order homogeneous terms of $\widetilde{R}_{(1,1)}(\widetilde{z}^{k_{(1,1)}},\widetilde{w}^{l_{(1,1)}})$ are precisely given by
\begin{align*}
{\left(\widetilde{R}_{(1,1)}\right)}_{E_{(1,1)}}(\widetilde{z}^{k_{(1,1)}},\widetilde{w}^{l_{(1,1)}})\text{.}
\end{align*}
This leads us to defining a (singular) change of coordinates $\Psi_{(1,1)}\colon\mathbb{C}^2\to\mathbb{C}^2$ by
\begin{align*}
\Psi_{(1,1)}(u,v)=(u^{{k_{(1,1)}}}+{{{\tau}_{(1,1)}}}v^{{l_{(1,1)}}},v^{l_{(1,1)}})\text{.}
\end{align*}
We set
\begin{align*}
R_{(1,1)}:=R\circ\Psi_{(1,1)}
\end{align*}
and once again consider the complex lines through $0$, along which the lowest-order homogeneous term of $R_{(1,1)}$ is harmonic (note that said lowest-order homogeneous term ``comes from'' $E_{(1,1)}$). If there is no such line we stop (if $1=l_1$) or we move on to considering the line $L_{(1,2)}$ (if $1<l_1$). Otherwise there will be finitely many such lines, say $L_{(2,1)},{\dots},L_{(2,c_{(1,1)})}$, where $c_{(1,1)}\geq{1}$ and $L_{(2,i)}\neq{L}_{(2,j)}$ for $i\neq{j}$. We add nodes $(2,1),{\dots},(2,c_{(1,1)})\in{V}$ corresponding to these lines, as well as edges $\{(2,1),(1,1)\},{\dots},\{(2,c_{(1,1)}),(1,1)\}\in{E}$ connecting these newly introduced nodes to $(1,1)$.

We want to iterate the procedure we just described.

\subsection{Some notation}

As is obvious from the steps carried out thus far, this iteration would lead to some very inconvenient indexing. In order to avoid this, we will introduce some notation that lets us work around this issue. First, we define a function
\begin{align*}
\mathcal{A}\colon{V}\setminus\{0\}\to{V}
\end{align*}
(the ``ancestor function'') that assigns to each node (except for the root) its ``immediate ancestor'', i.e.\ the second node on the uniquely determined shortest path $((m,n),{\dots},(0,0))$ to the root: $\mathcal{A}(m,n)=(m-1,j)$ for the (uniquely determined) $j$ with $\{(m,n),(m-1,j)\}\in{E}$.\\
Secondly, for a node $(m,n)\in{V}$, we denote the set of all nodes having $(m,n)$ as immediate ancestor as $\mathcal{C}(m,n)$ (the ``children set''):
\begin{align*}
\mathcal{C}(m,n)=\{(m+1,l)\in{V}\colon\mathcal{A}(m+1,l)=(m,n)\}=\mathcal{A}^{-1}(\{(m,n)\})\text{.}
\end{align*}

\subsection{Setting up the graph} \label{sub:graph}
We now carry out the construction of the graph indicated above in a more formal manner. We initialize the rooted undirected tree $G=(V,E)$ with $E=\emptyset$ and $V=\{(0,0)\}$. We also introduce a set $D$, the set of nodes that have been ``dealt with''; we start with $D=\emptyset$. We set $R_{(0,0)}:=R$ and $\Psi_{(0,0)}:=\operatorname{id}$ and $k_{(0,0)}=l_{(0,0)}=1$.

If the lowest order homogeneous term of $R_{(0,0)}$ is not harmonic along any complex line through $0$, we add $(0,0)$ to $D$. Otherwise, as mentioned previously, said term will be harmonic along only finitely many complex lines through $0$, say $L_{(1,1)},\dots{},L_{(1,l_1)}$, where $l_1\geq{1}$ and $L_{(1,i)}\neq{L_{(1,j)}}$ for $i\neq{j}$. We add nodes $(1,1),\dots{},(1,{l_1})\in{V}$ corresponding to these lines, as well as edges connecting each of these newly introduced nodes to the root, i.e.\ $\{(1,1),(0,0)\},{\dots},\{(1,{l_1}),(0,0)\}\in{E}$. After having introduced these new nodes and edges, we consider $(0,0)$ to be ``dealt with'', so we add $(0,0)$ to $D$.

We now iterate the following procedure:

If the set $V\setminus{D}$ is nonempty (i.e.\ there exists a node that has not been ``dealt with''), we do the following: pick the node $(m,n)\in{V}\setminus{D}$ that is minimal with respect to the lexicographical order.\\
The node $(m,n)$ comes from a complex line through $0$,
\begin{align*}
L_{(m,n)}=\{(z,w)\in\mathbb{C}^2\colon{z-{\tau}_{(m,n)}w=0}\}\text{,}
\end{align*}
along which the lowest-order homogeneous term of $R_{\mathcal{A}(m,n)}$ is harmonic. The lowest-order homogeneous term of the real-valued polynomial $\widetilde{R}_{(m,n)}$ given by
\begin{align*}
\widetilde{R}_{(m,n)}(\widetilde{z},\widetilde{w})={R_{\mathcal{A}(m,n)}}(\widetilde{z}+{\tau}_{(m,n)}\widetilde{w},\widetilde{w})
\end{align*}
is harmonic along the complex line $\{(\widetilde{z},\widetilde{w})\in\mathbb{C}^2\colon\widetilde{z}=0\}$. We consider the Newton diagram of $\widetilde{R}_{(m,n)}$. If there exists no extreme edge with slope $<-1$, we add $(m,n)$ to $D$ and go back to the beginning of the iteration. Otherwise let $E_{(m,n)}$ be the extreme edge with the {\emph{largest}} slope among all extreme edges with slope $<-1$ (Caution: this is now an extreme edge in a Newton diagram and {\emph{not}} an edge of the graph). We then find positive integers $k_{(m,n)},l_{(m,n)}$ with $\gcd(k_{(m,n)},l_{(m,n)})=1$, such that the lowest-order homogeneous terms of $\widetilde{R}_{(m,n)}(\widetilde{z}^{k_{(m,n)}},\widetilde{w}^{l_{(m,n)}})$ are precisely given by
\begin{align*}
{\left(\widetilde{R}_{(m,n)}\right)}_{E_{(m,n)}}(\widetilde{z}^{k_{(m,n)}},\widetilde{w}^{l_{(m,n)}})\text{.}
\end{align*}
This leads us to defining a (singular) change of coordinates $\Psi_{(m,n)}\colon\mathbb{C}^2\to\mathbb{C}^2$ by
\begin{align*}
\Psi_{(m,n)}(u,v)=(u^{{k_{(m,n)}}}+{{{\tau}_{(m,n)}}}v^{{l_{(m,n)}}},v^{l_{(m,n)}})\text{.}
\end{align*}
We set
\begin{align*}
R_{(m,n)}:={R_{\mathcal{A}(m,n)}}\circ\Psi_{(m,n)}
\end{align*}
and once again consider the complex lines through $0$, along which the lowest-order homogeneous term of $R_{(m,n)}$ is harmonic (note that said lowest-order homogeneous term ``comes from'' $E_{(m,n)}$). If there is no such line, we add $(m,n)$ to $D$ and go back to the beginning of the iteration. Otherwise there is a finite positive number of such lines, say $c_{(m,n)}$. We set
\begin{align*}
b_{(m,n)}:=\begin{cases}
\max\{j\in\mathbb{Z}\colon{(m+1,j)\in{V}}\} & \text{ if }(m+1,1)\in{V}\text{,}\\
0 & \text{ otherwise.}
\end{cases}
\end{align*}
We now name these lines
\begin{align*}
L_{(m+1,b_{(m,n)}+1)},\dots{},L_{(m+1,b_{(m,n)}+c_{(m,n)})}\text{,}
\end{align*}
and add nodes
\begin{align*}
{(m+1,b_{(m,n)}+1)},\dots{},{(m+1,b_{(m,n)}+c_{(m,n)})}\in{V}\text{,}
\end{align*}
corresponding to these lines, as well es edges
\begin{align*}
\{{(m+1,b_{(m,n)}+1)},(m,n)\},\dots{},\{{(m+1,b_{(m,n)}+c_{(m,n)})},(m,n)\}\in{E}\text{,}
\end{align*}
connecting these newly introduced nodes to the node $(m,n)$. We now add $(m,n)$ to $D$ and go back to the beginning of the iteration. It follows from \cite{FS:alg}, that $V\setminus{D}$ will be empty after {\emph{finitely many}} steps. This completes the construction of the graph $G$.

\subsection{Definition of $\Phi$}

Let $A\gg{0}$ be a large enough constant (to be made precise). We will, for each node $(m,n)$, define a function $\mathcal{D}_{(m,n)}$ and set
\begin{align*}
{\Phi}({\xi},z,w):={\xi}-A\cdot\mathcal{D}_{(0,0)}(z,w)\text{.}
\end{align*}
We will do so using a kind of ``backwards induction'', where we work our way from the leaves of the tree towards the root. More precisely, $\mathcal{D}_{(m,n)}$ will be determined by the functions associated to the nodes in the children set $\mathcal{C}(m,n)$ of $(m,n)$.

We start by defining $\mathcal{D}_{(m,n)}$ for a node $(m,n)$, whose children set is empty (note that this is equivalent to saying that $(m,n)$ is a leaf, unless $V=\{(0,0)\}$, in which case of course $(m,n)=(0,0)$). We consider two separate cases.

The first case is the case where $(m,n)\neq{(0,0)}$ and there exists no extreme edge with slope $<-1$ in the Newton diagram of $\widetilde{R}_{(m,n)}$. Looking at the construction of $G$, we see that ${\Psi}_{(m,n)}$ has not been defined in this case. We define ${\Psi}_{(m,n)}$ by
\begin{align*}
{\Psi}_{(m,n)}(u,v)=(u+{\tau}_{(m,n)}v,v)\text{,}
\end{align*}
and $\mathcal{D}_{(m,n)}$ by
\begin{align*}
\mathcal{D}_{(m,n)}(u,v)=|v|^{2L\cdot{l}_{\mathcal{A}(m,n)}\cdot{l}_{\mathcal{A}^{\circ{2}}(m,n)}\cdot\dots\cdot{l}_{\mathcal{A}^{\circ{m-1}}(m,n)}}\text{,}
\end{align*}
where $\mathcal{A}^{\circ{j}}=\mathcal{A}\circ\dots\circ\mathcal{A}$ with $j$ copies of $\mathcal{A}$ and $2L$ is the type at $0$. It should be noted that the product ${l}_{\mathcal{A}(m,n)}\cdot{l}_{\mathcal{A}^{\circ{2}}(m,n)}\cdot\dots\cdot{l}_{\mathcal{A}^{\circ{m-1}}(m,n)}$ is the empty product if $m<2$.

Now, still in the setting where $\mathcal{C}(m,n)$ is empty, we consider the case where one of the following two assertions is true:
\begin{itemize}
	\item $(m,n)={(0,0)}$
	\item $(m,n)\neq{(0,0)}$ and there {\emph{does}} exist an extreme edge with slope $<-1$ in the Newton diagram of $\widetilde{R}_{(m,n)}$.
\end{itemize}
Note that $\Psi_{(m,n)}$ was already defined in the construction of $G$ in this case. We define
\begin{align*}
\mathcal{D}_{(m,n)}(u,v)={\Vert{(u,v)}\Vert}^{2d(m,n)}\text{,}
\end{align*}
where $2d(m,n)$ is the degree of the lowest-order homogeneous term of $R_{(m,n)}$.

Finally, we consider a node $(m,n)$ with $\mathcal{C}(m,n)\neq\emptyset$. Then the lowest-order homogeneous term of $R_{(m,n)}$ is harmonic precisely along the following complex lines through $0$:
\begin{align*}
L_{(m+1,l)}\text{, where }(m+1,l)\in\mathcal{C}(m,n)\text{.}
\end{align*}
We choose a partition of unity
\begin{align*}
\left({\chi_{(m+1,l)}}\right)_{(m+1,l)\in\mathcal{C}(m,n)}
\end{align*}
with respect to conical neighborhoods of the $L_{(m+1,l)}$'s such that the ${{\chi}_{(m+1,l)}}$'s are homogeneous of degree $0$ and set:
\begin{align*}
\mathcal{D}_{(m,n)}(u,v)=\sum_{(m+1,l)\in\mathcal{C}(m,n)}\chi & _{(m+1,l)}(u,v)\\
\cdot\bigg(
& |u-{\tau}_{(m+1,l)}v|^{2d(m,n)}\\
& +|u-{\tau}_{(m+1,l)}v|^{2d(m,n)-2q(m,n)}|v|^{2q(m,n)}\\
& +\mathcal{D}_{(m+1,l)}({\Psi_{(m+1,l)}}^{-1}(u,v))\bigg)
\end{align*}
Here, $2d(m,n)$ is the degree of the lowest-order homogeneous term of $R_{(m,n)}$ and $2q(m,n)$ is the largest degree in $v,\overline{v}$ attained in the extreme set of the Newton diagram of $R_{(m,n)}$ corresponding to slope $-1$ (this can be an extreme point or an extreme edge). Furthermore, we point out that it is not a problem that the coordinate changes $\Psi_{(m+1,l)}$ are singular in general, since we are multiplying with an appropriate cut-off function $\chi_{(m+1,l)}$.

\section{Setting up $P_1$, $P_2$ and $P_3$} \label{section:ps}
%{\color{red} Have to change construction of $\Phi$ in order for the desired $P_2$ and $P_3$ to exist!!}
As mentioned previously, we choose $P_1\equiv 1$. Furthermore, $P_2$ and $P_3$ will only depend on $z$ and $w$, i.e.\ not on $\xi$. Our goal is to split $F$, $\Phi = \xi -F$, into terms divisible by $z$ and terms divisible by $w$ to obtain $P_2$ and $P_3$. Up to compositions of singular coordinate changes and multiplication with products of cut off functions, $F$ is a sum of terms of the following forms:
\begin{itemize}
	\item{$|v|^{2L\cdot{l}_{\mathcal{A}(m,n)}\cdot{l}_{\mathcal{A}^{\circ{2}}(m,n)}\cdot\dots\cdot{l}_{\mathcal{A}^{\circ{m-1}}(m,n)}}$,}
	\item{${\Vert{(u,v)}\Vert}^{2d(m,n)}$,}
	\item{$|u-{\tau}_{(m+1,l)}v|^{2d(m,n)}+|u-{\tau}_{(m+1,l)}v|^{2d(m,n)-2q(m,n)}|v|^{2q(m,n)}$.} 
\end{itemize}
Hence it is enough to treat each of these terms separately, while of course accounting for the singular coordinate changes.

We now fix a node $(m,n)$ for the remainder of this section. We start with the first term. Noting that, in the corresponding coordinates, we have $v^{{l}_{\mathcal{A}(m,n)}\cdot{l}_{\mathcal{A}^{\circ{2}}(m,n)}\cdot\dots\cdot{l}_{\mathcal{A}^{\circ{m-1}}(m,n)}}=w$, we readily decompose as follows:
\begin{align*}
& |v|^{2L\cdot{l}_{\mathcal{A}(m,n)}\cdot{l}_{\mathcal{A}^{\circ{2}}(m,n)}\cdot\dots\cdot{l}_{\mathcal{A}^{\circ{m-1}}(m,n)}}\\
=& z\cdot{0}+w\cdot{v^{(L-1)\cdot{l}_{\mathcal{A}(m,n)}\cdot{l}_{\mathcal{A}^{\circ{2}}(m,n)}\cdot\dots\cdot{l}_{\mathcal{A}^{\circ{m-1}}(m,n)}}}\overline{v^{L\cdot{l}_{\mathcal{A}(m,n)}\cdot{l}_{\mathcal{A}^{\circ{2}}(m,n)}\cdot\dots\cdot{l}_{\mathcal{A}^{\circ{m-1}}(m,n)}}}\text{.}
\end{align*}

In order to deal with the remaining two terms, we simplify notation a bit: we set $d_m :=d(m,n)$ and, when dealing with the last term, $q_m:=q(m,n)$. Furthermore we write $V_m :=(m,n)$, $V_{m-1}:=\mathcal{A}(m,n)$, \dots, $V_1:=\mathcal{A}^{\circ{m-1}}(m,n)$ and of course $V_0 :=(0,0)$. We denote the coordinates corresponding to $V_j$ as $(z_j,w_j)$ and let $\tau_j :=\tau_{V_j}$; in particular we have $(u,v)=(z_m , w_m )$ and $(z,w)=(z_0,w_0)$. The exponents from the coordinate changes are denoted as $\alpha_j:=k_{V_j}$ and $\beta_j:=l_{V_j}$, i.e.\ we have $w_{j}^{\beta_j}=w_{j-1}$ and $z_{j}^{\alpha_j}=z_{j-1}-\tau_{j} w_{j-1}$ for $j=1,\dots ,m$.

For the remaining two terms we notice that, away from $(0,0)$, they can be trivially rewritten as:
\begin{itemize}
    \item $\frac{{\Vert{(z_m,w_m)}\Vert}^{2d_m}}{|z_m|^{2d_m}+|w_m|^{2d_m}}\cdot{(|z_m|^{2d_m}+|w_m|^{2d_m})}$,
    \item $\frac{|z_m-{\tau}_{(m+1,l)}w_m|^{2d_m}+|z_m-{\tau}_{(m+1,l)}w_m|^{2d_m -2q_m}|w_m|^{2q_m}}{|z_m|^{2d_m}+|w_m|^{2d_m}}\cdot{(|z_m|^{2d_m}+|w_m|^{2d_m})}$.
\end{itemize}
But, away from $(0,0)$, both of these fractions are smooth bounded functions taking values in the non-negative reals. So, away from $(0,0)$, both of the remaining terms are of the form
\begin{align*}
    f_m\cdot (|z_m|^{2d_m}+|w_m|^{2d_m}) \text{,}
\end{align*}
for some smooth bounded function $f_m$, defined away from $(0,0)$ and taking values in $\mathbb{R}_{\geq 0}$. Of course such a function does not necessarily extend continuously to $(0,0)$, but the product of such a function with something small enough will, e.g. $f_m\cdot z_m$ or $f_m\cdot \overline{w_m}$. That is the idea we will use in the last step to obtain the desired splitting.  

Since the cut-off functions from the construction of $\Phi$ of course also occur in the expressions for $P_2$ and $P_3$, we again do not have to worry about the coordinate changes being singular. Because of this, we will ignore the singularity of the coordinate changes for the remainder of this section. We have, away from $(0,0)$:
\begin{align*}
    f_m\cdot (|z_m|^{2d_m} & +|w_m|^{2d_m})\\
    = & f_m\cdot |z_{m-1}-\tau_m w_{m-1}|^{\frac{2d_m}{\alpha_m}}+f_m\cdot |w|^{\frac{2d_m}{\beta_m\cdot\dots\cdot\beta_1}}\\
    = & f_m\cdot \frac{|z_{m-1}-\tau_m w_{m-1}|^{\frac{2d_m}{\alpha_m}}}{|z_{m-1}|^{\frac{2d_m}{\alpha_m}}+|w_{m-1}|^{\frac{2d_m}{\alpha_m}}}\cdot{(|z_{m-1}|^{\frac{2d_m}{\alpha_m}}+|w_{m-1}|^{\frac{2d_m}{\alpha_m}})}+f_m\cdot |w|^{\frac{2d_m}{\beta_m\cdot\dots\cdot\beta_1}}\\
    = & f_m\cdot f_{m-1}\cdot{(|z_{m-1}|^{\frac{2d_m}{\alpha_m}}+|w_{m-1}|^{\frac{2d_m}{\alpha_m}})}+f_m\cdot |w|^{\frac{2d_m}{\beta_m\cdot\dots\cdot\beta_1}}\\
    = & f_m\cdot f_{m-1}\cdot{|z_{m-1}|^{\frac{2d_m}{\alpha_m}}+f_m\cdot f_{m-1}\cdot |w|^{\frac{2d_m}{\alpha_m\cdot\beta_{m-1}\cdot\dots\cdot\beta_1}}}+f_m\cdot |w|^{\frac{2d_m}{\beta_m\cdot\dots\cdot\beta_1}}\text{,}
\end{align*}
where $f_{m-1}$ is again some smooth bounded function, defined away from $(0,0)$ and taking values in $\mathbb{R}_{\geq 0}$. Continuing inductively, we find smooth bounded functions $f_{m-2},\dots ,f_0$, defined away from $(0,0)$ and taking values in $\mathbb{R}_{\geq 0}$, such that:
\begin{align*}
    f_m\cdot (|z_m|^{2d_m}+|w_m|^{2d_m})= & f_m\cdot\dots\cdot f_0\cdot |z|^{\frac{2d_m}{\alpha_m\cdot\dots\cdot\alpha_1}}\\
    & + f_m\cdot\dots\cdot f_0\cdot |w|^{\frac{2d_m}{\alpha_m\cdot\dots\cdot\alpha_1}}\\
    & + \dots\\
    & + f_m\cdot\dots\cdot f_j\cdot |w|^{\frac{2d_m}{\alpha_m\cdot\dots\cdot\alpha_{j+1}\cdot\beta_j\cdot\dots\cdot\beta_1}}\\
    & + \dots\\
    & +f_m\cdot |w|^{\frac{2d_m}{\beta_m\cdot\dots\cdot\beta_1}}\text{.}
\end{align*}
This implies that $f_m\cdot (|z_m|^{2d_m}+|w_m|^{2d_m})$ is a (finite) sum of terms of the form
\begin{align*}
    g\cdot |x|^r\text{,}
\end{align*}
where $r$ is a positive real number, $x$ can be either $z$ or $w$, and $g$ is again some smooth bounded function, defined away from $(0,0)$ and taking values in $\mathbb{R}_{\geq 0}$. If $r>1$, then we can write
\begin{align*}
    g\cdot |x|^r=\left(g\cdot\frac{|x|}{x}\cdot |x|^{r-1}\right)\cdot x\text{.}
\end{align*}
Now $g|x|/x$ is smooth and bounded away from $(0,0)$, so (since $r-1>0$), the function $g\cdot\frac{|x|}{x}\cdot |x|^{r-1}$ extends continuously with value $0$ to $(0,0)$. Then, depending on whether $x$ is $z$ or $w$, we absorb the corresponding term into $P_2$ respectively $P_3$.

The only thing left to do is to show that the occurring exponents are larger than $1$, i.e.\ we have to show that ${2d_m}>{\alpha_m\cdot\dots\cdot\alpha_{j+1}\cdot\beta_j\cdot\dots\cdot\beta_1}$ for $j=0,1,{\dots},m$. But, since all $\alpha_i$, $\beta_i$ come from extreme edges with slope $\leq -1$, we clearly have $\alpha_i\geq\beta_i\geq 1$ for all $i$, i.e.\ we only have to show that $2d_m > \alpha_m\cdot\dots\cdot\alpha_1$. This, however, follows immediately by tracing through the algorithm described in the previous section.

\section{Developing plurisubharmonic weights} \label{section:weights}
In the next section, we will need weights coming from the algorithm in the use of H{\"o}rmander's theorem. We develop these weights in this section.

For each node $(m,n)$ of $G$, whose children set $\mathcal{C}(m,n)$ is empty, we define a function $\rho_{(m,n)}$, which will appear in the definition of the weight for the $(0,2)$-form $\omega$. We fix such a node $(m,n)$ for the remainder of this section. Much like in the definition of $\Phi$, we trace our way back from $(m,n)$ to the root and add terms along the way. We will have
\begin{align*}
\rho_{(m,n)}(z,w,{\xi})=\log\left({|{\xi}|+|{\xi}|^2+H_{(m,n)}(z,w)}\right)\text{,}
\end{align*}
for a real valued function $H_{(m,n)}\geq{0}$ that will be described below. If $(m,n)=(0,0)$, then $V=\{(0,0)\}$ and we set $H_{(0,0)}:=\mathcal{D}_{(0,0)}$. So assume $(m,n)\neq{(0,0)}$ for the remainder of this section, i.e.\ $V\neq{(0,0)}$.

The function $H_{(m,n)}$ will look like the function $\mathcal{D}_{(0,0)}$, except for the fact that all the occurring cut-off functions are replaced by either $1$ or $0$, depending on whether the node of consideration lies on the uniquely determined shortest path from $(m,n)$ to $(0,0)$. In other words, $H_{(m,n)}$ is defined like $\mathcal{D}_{(0,0)}$, but we do not have cut-off functions and only take the nodes corresponding to the shortest path between $(m,n)$ and the root in $G$. In the language of \cite{FS:alg}, this path in $G$ corresponds to a mother curve. Since we do not have cut-off functions to take care of the fact that the coordinate changes are singular, we sum over all the preimages and normalize to prevent problems induced by multiplicity. We carry this out formally:

For each node of the form $\mathcal{A}^{\circ{j}}(m,n)$, where $j\in\{0,1,{\dots},m\}$, we define a function $\mathcal{E}_{\mathcal{A}^{\circ{j}}(m,n)}^{(m,n)}$ and set $H_{(m,n)}:=\mathcal{E}_{\mathcal{A}^{\circ{m}}(m,n)}^{(m,n)}=\mathcal{E}_{(0,0)}^{(m,n)}$.

We set
\begin{align*}
\mathcal{E}_{\mathcal{A}^{\circ{0}}(m,n)}^{(m,n)}=\mathcal{E}_{(m,n)}^{(m,n)}:=\mathcal{D}_{(m,n)}\text{,}
\end{align*}
(recall that $(m,n)$ is a leaf) and for $j\in\{0,{\dots},m-1\}$. Finally we define $\mathcal{E}_{\mathcal{A}^{\circ{j+1}}(m,n)}^{(m,n)}(u,v)$ using an average as follows
\begin{align*}
\mathcal{E}_{\mathcal{A}^{\circ{j+1}}(m,n)}^{(m,n)} & (u,v):=\\
& |u-{\tau}_{\mathcal{A}^{\circ{j}}(m,n)}v|^{2d(\mathcal{A}^{\circ{j+1}}(m,n))}\\
& +|u-{\tau}_{\mathcal{A}^{\circ{j}}(m,n)}v|^{2d(\mathcal{A}^{\circ{j+1}}(m,n))-2q(\mathcal{A}^{\circ{j+1}}(m,n))}|v|^{2q(\mathcal{A}^{\circ{j+1}}(m,n))}\\
& +\frac{1}{\operatorname{card}({\Psi_{\mathcal{A}^{\circ{j}}(m,n)}}^{-1}\{(u,v)\})}\sum_{(\widetilde{u},\widetilde{v})\in{\Psi_{\mathcal{A}^{\circ{j}}(m,n)}}^{-1}\{(u,v)\}}\mathcal{E}_{\mathcal{A}^{\circ{j}}(m,n)}^{(m,n)}(\widetilde{u},\widetilde{v})\text{.}
\end{align*}

\section{$L^2$ estimates for the $(2,0)$-form $\omega$} \label{section:omegaestimate}

We start with a lemma that will simplify the estimates.
\begin{lemma}
	\label{onesummandenough}
	Let $m$ be a positive integer and let $T,S\in\mathbb{C}^{m\times{m}}$ be Hermitian matrices, such that $S$ is positive semidefinite and $T$ is positive definite. Then we have for all $v\in\mathbb{C}^m$:
	\begin{align*}
	\overline{v}^t{{(T+S)}^{-1}}{v}\leq\overline{v}^t{{T}^{-1}}{v}\text{.}
	\end{align*}
\end{lemma}
\begin{proof}
	This follows by writing down a Cholesky decomposition for $T^{-1}$ and calculating.
\end{proof}	

Now we want to solve the equation $\overline\partial u  = \omega$ in an $L^2$ space using the following theorem that allows us to gain more regularity.

\begin{thm} (H\"ormander, Demailly)
Let $\rho$ be a plurisubharmonic function on $D\subset \mathbb C^3$, pseudoconvex, $v$ a $\overline\partial-$
closed $(0,2)-$form. Then there exists a $(0,1)-$form such that $\overline\partial u=v$ and

$$
\int_D |u|^2 e^{-\rho} \leq C \int_D <A^{-1}v, v> e^{-\rho}\text{,}
$$
where 
\[
A=
\begin{bmatrix}
\rho_{{\xi}\overline{{\xi}}} +  \rho_{{z}\overline{{z}}}              & \rho_{{w}\overline{{z}}}  & -\rho_{{w}\overline{{\xi}}}  \\
\rho_{{z}\overline{{w}}}  & \rho_{{\xi}\overline{{\xi}}} +\rho_{{w}\overline{{w}}}  & \rho_{{z}\overline{{\xi}}} \\
-\rho_{{\xi}\overline{{w}}}  & \rho_{{\xi}\overline{{z}}} & \rho_{{z}\overline{{z}}}+ \rho_{{w}\overline{{w}}} 
\end{bmatrix}\text{.}
\]

\end{thm}

In the setting of the theorem, an explicit calculation gives:
\begin{align*}
A^{-1}=\frac{1}{\det A}\cdot M\text{,}
\end{align*}
where
\begin{align*}
M& =
\begin{bmatrix}
(\rho_{{\xi}\overline{{\xi}}}+\rho_{{w}\overline{{w}}})
(\rho_{{z}\overline{{z}}}+\rho_{{w}\overline{{w}}})-|\rho_{{\xi}\overline{{z}}}|^{2} &  -\rho_{{z}\overline{{w}}}(\rho_{{z}\overline{{z}}}+\rho_{{w}\overline{{w}}})-\rho_{{\xi}\overline{{w}}}\rho_{{z}\overline{{\xi}}} & \rho_{{z}\overline{{w}}}\rho_{{\xi}\overline{{z}}}+\rho_{{\xi}\overline{{w}}}(\rho_{{\xi}\overline{{\xi}}}+\rho_{{w}\overline{{w}}})\\
-\rho_{{w}\overline{{z}}}(\rho_{{z}\overline{{z}}}+\rho_{{w}\overline{{w}}})-\rho_{{\xi}\overline{{z}}}\rho_{{w}\overline{{\xi}}} &
(\rho_{{\xi}\overline{{\xi}}}+\rho_{{z}\overline{{z}}})
(\rho_{{z}\overline{{z}}}+\rho_{{w}\overline{{w}}})-|\rho_{{\xi}\overline{{w}}}|^{2} &  -\rho_{{\xi}\overline{{z}}}(\rho_{{\xi}\overline{{\xi}}}+\rho_{{z}\overline{{z}}})-\rho_{{\xi}\overline{{w}}}\rho_{{w}\overline{{z}}}\\
\rho_{{w}\overline{{z}}}\rho_{{z}\overline{{\xi}}}+\rho_{{w}\overline{{\xi}}}(\rho_{{\xi}\overline{{\xi}}}+\rho_{{w}\overline{{w}}}) & 
-\rho_{{z}\overline{{\xi}}}(\rho_{{\xi}\overline{{\xi}}}+\rho_{{z}\overline{{z}}})-\rho_{{w}\overline{{\xi}}}\rho_{{z}\overline{{w}}} & 
(\rho_{{\xi}\overline{{\xi}}}+\rho_{{z}\overline{{z}}})
(\rho_{{\xi}\overline{{\xi}}}+\rho_{{w}\overline{{w}}})-|\rho_{{z}\overline{{w}}}|^{2}
\end{bmatrix}\text{,}
\end{align*}
and
\begin{align*}
\det A& =(\rho_{{\xi}\overline{{\xi}}}+\rho_{{z}\overline{{z}}})\left[\rho_{{\xi}\overline{{\xi}}}\rho_{{z}\overline{{z}}}-|\rho_{{z}\overline{{\xi}}}|^{2} \right] +
(\rho_{{w}\overline{{w}}}+\rho_{{z}\overline{{z}}})\left[\rho_{{z}\overline{{z}}}\rho_{{w}\overline{{w}}}-|\rho_{{z}\overline{{w}}}|^{2} \right]+
(\rho_{{\xi}\overline{{\xi}}}+\rho_{{w}\overline{{w}}})\left[\rho_{{\xi}\overline{{\xi}}}\rho_{{w}\overline{{w}}}-|\rho_{{w}\overline{{\xi}}}|^{2} \right]\\
& \phantom{=}+{2} \text{ Levi Det}(\rho)+L_{{\xi}{z}}(\rho, (\rho_{{w}\overline{{z}}},\rho_{\overline{{\xi}}{w}}))+L_{{\xi}{w}}(\rho, (\rho_{{w}\overline{{z}}},\rho_{\overline{{\xi}}{z}}))+L_{{z}{w}}(\rho, (\rho_{{\xi}\overline{{w}}},\rho_{\overline{{\xi}}{z}}))\text{.}
\end{align*}

We want to apply this theorem with $\Psi_0$, $\Omega_0^{**}$, $\omega$ in the roles of $\rho$, $D$, $v$ respectively, where $\Omega_0^{**}$ is pseudoconvex and contains $\overline{\Omega}\setminus\{0\}$. Moreover, locally, it is an intermediate bumping in the sense that $\overline{\Omega}\setminus\{0\}\subseteq\Omega_0^{**}\subseteq\Omega_0^{*}$ and both inclusions denote a bumping to the type of $\Omega$ at $0$. More precisely, $\Omega_0^{**}$ is obtained by ``subtracting'' half the bumping function from the defining function of $\Omega$.

Furthermore,
\begin{align*}
\omega =& 2\frac{\overline{\partial}{P_2}\wedge\overline{\partial}{P_3}}{{\Phi}^3}\text{,}\\
\Psi_0 =& -\left({\frac{1}{k}+\epsilon{J}}\right)\cdot\log{(\operatorname{dist}({\cdot},\operatorname{b}\Omega_0^{*}))}+d(|{\xi}|^2+|z|^2+|w|^2)\\
& +\sum_{(m,n)\in{V}\colon\mathcal{C}(m,n)=\emptyset}^{}{\epsilon\rho_{(m,n)}}\text{,}
\end{align*}
where $J$ is the number of nodes with empty children set (i.e.\ the number of leaves, if $V\neq\{(0,0)\}$) and $d,{\epsilon}>0$ are very small. It should be pointed out that the term involving $d$ is only included to ensure invertibility resp.\ positivity in the appropriate places, and will not play a big role in the following estimates. Recall furthermore that for the nodes $(m,n)$ with empty children set we have
\begin{align*}
\rho_{(m,n)}({\xi},z,w)=\log\left({|{\xi}|+|{\xi}|^2+H_{(m,n)}(z,w)}\right)\text{.}
\end{align*}

If
\begin{align*}
\Psi_0={\rho}+\Psi_0'\text{,}
\end{align*}
where $\rho$ is strictly plurisubharmonic and $\Psi_0'$ is plurisubharmonic, then the matrix $A_\rho$ is a positive definite Hermitian matrix and $A_{\Psi_0}-A_\rho$ is a positive semidefinite Hermitian matrix. Lemma \ref{onesummandenough} and a calculation then immediately give:
\begin{align*}
|<{A_{\Psi_0}}^{-1}{\omega},{\omega}>| e^{-\Psi_0} \leq |\omega|^2 e^{-\Psi_0} \cdot (\operatorname{I}+\operatorname{II}+\operatorname{III})\text{,}
\end{align*}
where
\begin{align*}
\operatorname{I} & =\frac{\rho_{\xi\overline\xi}}{\rho_{\xi\overline\xi}\rho_{z\overline{z}}-|\rho_{\xi\overline{z}}|^2}\text{,}\\
\operatorname{II} & =\frac{\rho_{\xi\overline\xi}}{\rho_{\xi\overline\xi}\rho_{w\overline{w}}-|\rho_{\xi\overline{w}}|^2}\text{,}\\
\operatorname{III} & =\frac{1}{\rho_{z\overline{z}}+\rho_{w\overline{w}}}\text{.}
\end{align*}
We have $\operatorname{III}\leq\operatorname{I}$ (since $\rho$ is strictly plurisubharmonic), so it suffices to estimate $\operatorname{I}$ and $\operatorname{II}$.

Now assume that $\rho$ is of the following form:
\begin{align*}
\rho ({\xi},z,w)=\epsilon\cdot\log\left({|{\xi}|+|{\xi}|^2+H(z,w)}\right)\text{,}
\end{align*}
where $H\geq 0$ is a smooth real-valued function with the property that $H_{z\overline z}$ and $HH_{z\overline{z}}-|H_z|^2$ (as well as the analogous expressions with $w$ instead of $z$) are non-negative. Setting $h({\xi},z,w):={|{\xi}|+|{\xi}|^2+H(z,w)}$, a direct computation gives:
\begin{align*}
\operatorname{I} & \lesssim\frac{h^2}{H_{z\overline{z}}\cdot (|{\xi}|+|{\xi}|^2+H)}+\frac{h^2}{HH_{z\overline{z}}-|H_z|^2}\\
& \lesssim\frac{h^2}{HH_{z\overline{z}}-|H_z|^2}\text{,}
\end{align*}
and analogously we of course get
\begin{align*}
\operatorname{II}\lesssim\frac{h^2}{HH_{w\overline{w}}-|H_w|^2}\text{,}
\end{align*}
where the occurring constants depend on $\epsilon$.

\subsection{Estimating the integral}
We are integrating over $\Omega_0^{**}$. We partition the domain of integration into finitely many sets in accordance with how $\Phi$ was defined.

Intuitively speaking, we do the following: if the lowest order homogeneous term of $R$ is not harmonic along any complex line through $0$, the partition is simply given by the domain itself. Otherwise, we remove small conical neighborhoods of the complex lines through $0$, along which $R$ is harmonic (namely $L_{(1,1)},\dots{},L_{(1,l_1)}$). The resulting set gives the first set of the partition.\\
Then, for every complex line $L_{(1,j)}$ such that $(1,j)$ is a leaf, the corresponding conical neighborhood gives a set of the partition. For every complex line $L_{(1,j)}$ such that $(1,j)$ is {\emph{not}} a leaf, the lowest order homogeneous term of $R_{(1,j)}=R\circ\Psi_{(1,j)}$ is harmonic along a finite (strictly positive) number of complex lines through $0$. Once again, we remove small conical neighborhoods of these lines and the resulting set contributes a set to the partition after adjusting for the change of coordinates $\Psi_{(1,j)}$. We continue in the obvious way and obtain the announced partition, which is finite, since the graph is finite. We denote this partition as
\begin{align*}
\Omega_0^{**}=\bigcup_{(m,n)\in V}\mathcal{S}{(m,n)}\text{,}
\end{align*}
where $\mathcal{S}(m,n)$ is the region corresponding to the node $(m,n)$. We estimate the integral by considering the regions corresponding to the nodes separately, i.e.\ it suffices to show that the integral
\begin{align*}
\int_{\mathcal{S}{(m,n)}}{|<{A_{\Psi_0}}^{-1}{\omega},} & {{\omega}>| e^{-\Psi_0}}\\
\lesssim & {\phantom{+}}\int_{\mathcal{S}{(m,n)}}\frac{|\overline{\partial}_{z,w}{P_2}\wedge\overline{\partial}_{z,w}{P_3}|^2}{|{\Phi}|^6} e^{-\Psi_0}\cdot\frac{h^2}{HH_{z\overline{z}}-|H_z|^2}\\
& +\int_{\mathcal{S}{(m,n)}}\frac{|\overline{\partial}_{z,w}{P_2}\wedge\overline{\partial}_{z,w}{P_3}|^2}{|{\Phi}|^6} e^{-\Psi_0}\cdot\frac{h^2}{HH_{w\overline{w}}-|H_w|^2}
\end{align*}
is finite for every node $(m,n)\in V$, where integration occurs with respect to the Lebesgue measure on $\mathbb{R}^6$. We start at the root $(0,0)$ and inductively work our way down to all the leaves.

\subsubsection{first step}
We start at the root of the tree, i.e.\ we look at a region in the $(z,w)$-plane (in the original coordinates), where small conical neighborhoods of the critical complex lines through $0$ have been removed (the complex lines through $0$, along which the lowest-order homogeneous term of $R$ is harmonic (finitely many)). If $(0,0)$ has empty children set, then there have not been any lines removed and $\Phi$ takes a particularly simple form. The estimates are much easier in this case, so we assume that $\mathcal{C}(0,0)\neq\emptyset$.\\
\\
We take $\rho$ to be {\emph{any}} of the $\rho_{(m,n)}$ (the choice will matter in the induction step, but not here). Considering that we avoid conical neighborhoods of the critical lines, we can ignore the terms contributed by other nodes of the graph in the estimates for the region $\mathcal{S}(0,0)$. We essentially have the following:
\begin{align*}
\rho & =\log\left({|{\xi}|+|{\xi}|^2+|z-{\tau}w|^{2k}+|z-{\tau}w|^{2j}|w|^{2k-2j}+\operatorname{remainder}}\right)\\
h &={|{\xi}|+|{\xi}|^2+|z-{\tau}w|^{2k}+|z-{\tau}w|^{2j}|w|^{2k-2j}+\operatorname{remainder}}\\
H & =|z-{\tau}w|^{2k}+|z-{\tau}w|^{2j}|w|^{2k-2j}+\operatorname{remainder}\text{,}
\end{align*}
where $z$ is bounded away from $\tau w$. Also
\begin{align*}
    P_2 & =  -A(z-\tau w)^{k-1}(\overline{z}-\overline{\tau} \overline{w})^k+\operatorname{remainder}\\
P_3 & = -A(-\tau (z-\tau w)^{k-1}(\overline{z}-\overline{\tau} w)^k
+|z-\tau w|^{2j_i}w^{k-j_i-1}\overline{w}^{k-j_i})+\operatorname{remainder},
\end{align*}

	where the respective remainders are also insignificant when computing derivatives. A calculation gives (here, $x$ can be either $z$ or $w$; note that we are in the first step):
	\begin{align*}
	|\overline{\partial}P_2|^2 & \sim |x|^{4k-4}\\
	HH_{z\overline{z}}-|H_z|^2 & = (k-j)^2|z-{\tau}w|^{2k+2j-2}|w|^{2k-2j}+\operatorname{remainder}\\
	& \sim |x|^{4k-2}\\
	HH_{w\overline{w}}-|H_w|^2 & = (k-j)^2|z-{\tau}w|^{2k+2j-2}|w|^{2k-2j-2}|z|^2+\operatorname{remainder}\\
	& \sim |x|^{4k-2}\text{.}
	\end{align*}
	
	Using this, we can estimate the integrand as follows (we introduce a small $\delta >0$ to kill a potential $\log$-term; note also that in the current region we have $|\Phi |\sim h$):
	\begin{align*}
	& \frac{|\overline{\partial}{P_2}\wedge\overline{\partial}{P_3}|^2}{|{\Phi}|^6} e^{-\Psi_0 (1+\delta )}\cdot\left(\frac{h^2}{HH_{z\overline{z}}-|H_z|^2}+\frac{h^2}{HH_{w\overline{w}}-|H_w|^2}\right)\\
	\lesssim & e^{-\Psi_0 (1+\delta )}\cdot \frac{h^2}{|x|^{4k-2}}\cdot\frac{|x|^{4k-4}|\overline{\partial}P_3|^2}{|{\Phi}|^6}\\
	\lesssim & e^{-\Psi_0 (1+\delta )}\cdot\frac{1}{|x|^2}\cdot\frac{|\overline{\partial}P_3|^2}{|{\Phi}|^4}\text{.}
	\end{align*}
	We have to integrate with respect to the form $d\xi\wedge d\overline{\xi}\wedge dz\wedge d\overline{z}\wedge dw\wedge d\overline{w}$.
	Roughly speaking, integrating with respect to $d\xi\wedge d\overline{\xi}$ turns $1/|\Phi |^4$ into $1/|\Phi |^2$ and integrating with respect to $dz\wedge d\overline{z}$ takes care of the $1/|x|^2$. What remains is to estimate
	\begin{align*}
	\int e^{-\Psi_0 (1+\delta )}\cdot\frac{|\overline{\partial}P_3|^2}{|{\Phi}|^2}dw\wedge d\overline{w}\text{,}
	\end{align*}
	which turns out to be finite, as desired.
	
	\subsubsection{Induction step}
	It remains to see what happens in the conical neighborhood of one of the lines. After the usual coordinate change, this region looks like the region from the first step. So this leads to an inductive procedure, where we dig our way down from the root of the tree all the way down to the leaves (with only the leaves needing special treatment, since the leaf-terms look a bit different).
	
	So consider the region $\mathcal{S}(m,n)$ for a node $(m,n)\neq{(0,0)}$. In the new coordinates, the region looks the same as the region from the first step, so if we can convince ourselves that also the integrand including the Jacobian from the change of coordinates looks the same as in the first step (in the new coordinates), then the integral can be estimated precisely as before. We denote the coordinates corresponding to the node $\mathcal{A}(m,n)$ as $(\xi ,z,w)$. We call the coordinate change $\phi$ and denote the new coordinates as $(z_1 , w_1)$; the coordinate $\xi$ of course stays the same. If $(m,n)$ is a leaf, then the occurring expressions take a simpler form, making the estimates much easier; so we assume that $(m,n)$ is not a leaf.
	\\
	
	We look at the integral involving the ${h^2}/({HH_{z\overline{z}}-|H_z|^2})$ term (the other term ($w$ instead of $z$) can be handled similarly, although the expression of that term in the new coordinates will look a bit more complicated).
	\\
	
	We take $\rho =\rho_{(m',n')}$, where $(m',n')$ is any leaf with the property that the current node $(m,n)$ lies on the uniquely determined shortest path from the root to $(m',n')$. With other words: the current node $(m,n)$ can be reached from $(m',n')$ by iterating the ancestor function $\mathcal{A}$. So we look at
	\begin{align*}
	\frac{|\overline{\partial}_{z,w}{P_2}\wedge\overline{\partial}_{z,w}{P_3}|^2}{|{\Phi}|^6} e^{-\Psi_0}\cdot\frac{h^2}{HH_{z\overline{z}}-|H_z|^2}\cdot |\det\phi '|^2
	\end{align*}
	in the new coordinates $(z_1 , w_1 )$, where the determinant factor comes from the change of variables. Since
	\begin{align*}
	|\overline{\partial}_{z,w}{P_2}\wedge\overline{\partial}_{z,w}{P_3}|^2\cdot |\det\phi '|^2=|\overline{\partial}_{z_1 ,w_1 }{P_2}\wedge\overline{\partial}_{z_1 ,w_1 }{P_3}|^2\text{,}
	\end{align*}
	we can estimate the integrand to be
	\begin{align*}
	\lesssim \frac{|\overline{\partial}_{z_1 ,w_1 }{P_2}|^2|\overline{\partial}_{z_1 ,w_1 }{P_3}|^2}{|{\Phi}|^6} e^{-\Psi_0}\cdot\frac{h^2}{HH_{z\overline{z}}-|H_z|^2}
	\end{align*}
	We want to make this look like in the first step. But recalling how $P_2$ and $P_3$ were chosen relative to $\Phi$, the expression $|\overline{\partial}_{z_1 ,w_1 }{P_2}|^2$ (resp.\ $|\overline{\partial}_{z_1 ,w_1 }{P_3}|^2$) is still missing a factor $\sim |z/z_1 |^2$ (resp.\ $\sim |w/w_1 |^2$) in order to look like in the previous step. Furthermore, we still need to express ${HH_{z\overline{z}}-|H_z|^2}$ in the new coordinates and we point out that the $1/k$ occurring in the definition of $\Psi_0$ still corresponds to the first step. But, using that $\partial w/\partial z_1=0$ and $z={z_1}^{\alpha}+{\tau}w_1^\beta$, we get:
	\begin{align*}
	\frac{1}{HH_{z\overline{z}}-|H_z|^2}=\frac{1}{HH_{z_1\overline{z_1}}-|H_{z_1}|^2}\cdot \left|\frac{\partial z}{\partial z_1}\right|^2=\frac{1}{HH_{z_1\overline{z_1}}-|H_{z_1}|^2}\cdot \left|\alpha{z_1}^{\alpha -1}\right|^2\text{.}
	\end{align*}
	So, in order to account for the $1/k$ term in the definition of $\Psi_0$ and the missing factor $|w/w_1 |^2\cdot |z/z_1 |^2$, we point out that (roughly speaking)
	\begin{align*}
	e^{1/k\log|\Phi |}\cdot \left|\alpha{z_1}^{\alpha -1}\right|^2=\left|\frac{w}{w_1}\right|^2\cdot\left|\frac{w_1}{w}\right|^2\cdot \left|\alpha{z_1}^{\alpha -1}\right|^2 \cdot e^{1/k\log|\Phi |}\sim \left|\frac{w}{w_1}\right|^2\cdot\left|\frac{z}{z_1}\right|^2\cdot e^{1/{k_1}\log|\Phi |}\text{,}
	\end{align*}
	where $2k_1$ is the degree of the lowest-order homogeneous term when expressing $R$ with respect to $z_1$ and $w_1$. Now the integrand is as in the first step and the estimate goes through the same way.

\subsection{Volume of polydiscs}
In Section \ref{section:01estimates} we will pass from $L^2$ estimates to pointwise estimates using subaveraging. With this in mind, we have to describe the volume of ``a large polydisc'' $Q(q)\subseteq\Omega_{0}^{**}$ centered at a boundary point $q\in\partial\Omega\setminus\{0\}$ close to $0$. This will be done with respect to the coordinates corresponding to the current region in the partition of the $(z,w)$-plane; $Q(q)$ is chosen to be a polydisc in those coordinates:

Recalling that
\begin{align*}
\Omega_0^{**}=\bigcup_{(m',n')\in V}\mathcal{S}{(m',n')}\text{,}
\end{align*}
where $\mathcal{S}(m',n')$ is the region corresponding to the node $(m',n')$, we find a node $(m,n)$, such that $q\in \mathcal{S}{(m,n)}$. In the coordinates $(\xi ,z_{(m,n)},w_{(m,n)})$ corresponding to $(m,n)$, we can fit a polydisc $Q(q)$ with
\begin{align*}
    \sqrt{\operatorname{Vol}(Q(q))}\sim (|z_{(m,n)}|+\norm{(\xi ,z,w)}^{2L})(|w_{(m,n)}|+\norm{(\xi ,z,w)}^{2L})|{\Phi}(q)|\text{,}
\end{align*}
as is obvious from the choice of partition.

\section{Passing from $L^2$ estimates to pointwise estimates for $v_1$, $v_2$, $v_3$} \label{section:01estimates}

The aim of this section is to give good local estimates for the functions $v_1$, $v_2$, $v_3$ appearing in the Koszul complex (see Section \ref{section:koszul}). Recall, we do this in the local coordinates where the domain $\Omega$ is given by
\begin{align*}
\Omega=\left\{\operatorname{Re}({\eta_1^0-z_1})+r(\eta_2^0-z_2,\eta_3^0-z_3)+s({\eta_1^0-z_1},\eta_2^0-z_2,\eta_3^0-z_3)<0\right\}.
\end{align*}

\begin{lemma} \label{lem:estimates}
    Let $p=(\eta_1^{0},\eta_2^{0},\eta_3^{0})\in \di \Omega$, $(z_1,z_2,z_3) \in \di \Omega$, $2L$ is the D'Angelo type, $2k$ is the hypersurface type (or Bloom-Graham type), and $\delta>0$ such that $\delta \ll \frac{1}{2k}$ and $\delta \ll \frac{1}{2L}$, then
    
\begin{align*}
|v_1(z_1,z_2,z_3)| & \leq C \cdot \frac{|\eta_3^{0}-z_3|}{M_1 M_2} \cdot \frac{1}{\Phi^{1+\frac{1}{2k}+\delta-\frac{1}{2L}}} \\
|v_2(z_1,z_2,z_3)| & \leq C \cdot \frac{|\eta_2^{0}-z_2|}{M_1 M_2}\cdot \frac{1}{\Phi^{1+\frac{1}{2k}+\delta-\frac{1}{2L}}}\\
|v_3(z_1,z_2,z_3)| & \leq C \cdot \frac{1}{M_1 M_2}\frac{1}{\Phi^{\frac{1}{2k}+\delta-\frac{1}{2L}}}
\end{align*}
where 
\begin{align*}
M_1 & = |z_{(m,n)}|+\norm{(\eta_1^0-z_1,\eta_2^0-z_2,\eta_3^0-z_3)}^{2L}\\
M_2 & = |w_{(m,n)}|+\norm{(\eta_1^0-z_1,\eta_2^0-z_2,\eta_3^0-z_3)}^{2L}
\end{align*}
and $z_{(m,n)}$ and $w_{(m,n)}$ reflect of the changes of coordinates in different zones as we approach an exceptional curve (see Subsection \ref{sub:graph}), $\Phi=\Phi(\xi, z, w)$ and $\xi=\eta_1^{0}-z_1$, $z=\eta_2^0-z_2$, and $w=\eta_3^0-z_3$
\end{lemma}

\begin{proof} We choose $\e, \tilde{e}>0$ so that $\e J+\frac{\tilde{\e}}{k}<\delta$. We have shown that we can solve $\overline\partial u=\omega$ on $\Omega_0^{**}$ such that
$$
\int_{\Omega_0^{**}} |u|^2 e^{-\psi_0}\leq C.
$$

Let $\kappa = \log{\left(|\xi| +|\xi|^2 +|z|^{2M}+|w|^{2M}\right)}$
where $M$ is larger than the degree of bumping from Diederich-Forn\ae ss (see \cite{DF}) bumping. Here we use the weight 
\begin{equation*}
\widetilde{\psi_0}= \psi_0 -\widetilde{\epsilon}\log \operatorname{dist}({\cdot},\partial\Omega_0^{*})+\widetilde{\epsilon}\kappa\text{.}
\end{equation*}

Define
\begin{align*}
\psi_1 & =\widetilde{\psi_0} +\log (|w|^2+\norm{(\xi ,z,w)}^{4L})+\log (\norm{(\xi ,z,w)}^{2})\\
\psi_2 & =\widetilde{\psi_0} +\log (|z|^2+\norm{(\xi ,z,w)}^{4L})+\log (\norm{(\xi ,z,w)}^{2})\\
\psi_3 & =\widetilde{\psi_0} +\log (|{\xi}|^2+\norm{(\xi ,z,w)}^{4L})+\log (\norm{(\xi ,z,w)}^{2}).
\end{align*}

Noting that
\bea
h_{1,2} & = -\frac{\overline\partial P_2}{\Phi^2}\\
h_{1,3} & = -\frac{\overline\partial P_3}{\Phi^2}\\
h_{2,3} & = \frac{P_3\overline\partial P_2-P_2\overline\partial P_3}{\Phi^2},\\
\eea
our next task is to find solutions with good estimates for
\bea
\overline\partial v_1 & = & -\frac{\overline\partial P_2}{\Phi^2}-w u=:s_1\\
\overline\partial v_2 & = & -\frac{\overline\partial P_3}{\Phi^2}-z u=:s_2\\
\overline\partial v_3 & = & \frac{P_3\overline\partial P_2-P_2\overline\partial P_3}{\Phi^2}-\xi u=:s_3.
\eea

Since $u\in L^2 (\psi_0 )$, it follows that $s_1\in L^2 (\psi_1 )$, $s_2\in L^2 (\psi_2 )$ and $s_3\in L^2 (\psi_3 )$. Recall H{\" o}rmander's $L^2$-estimates:

\begin{thmnonum}
	Let $\rho$ be a plurisubharmonic function on $D\subset \mathbb C^3$, pseudoconvex, $s$ is a $\overline\partial-$
	closed $(0,1)-$form. Then there exists a $(0,1)-$form $v$ such that $\overline\partial v=s$ and
	
	$$
	\int_D |v|^2 e^{-\rho} \leq C \int_D <A^{-1}s, s> e^{-\rho}\text{,}
	$$
	where $A$ is the Complex Hessian Matrix of $\rho$.
\end{thmnonum}
We apply the theorem with $s_j$, $\psi_j$, $\Omega_0^{**}$ in the roles of $s$, $\rho$, $D$ respectively. With regards to finiteness of the integral on the right hand side we note that, using Lemma \ref{onesummandenough}:
\begin{align*}
    <A_{\psi_j}^{-1}s_j, s_j> & \leq \frac{1}{\widetilde{\e}}<A_{\kappa}^{-1}s_j, s_j>\\
    & \leq \frac{C}{\widetilde{\epsilon}}\norm{s}^2\norm{(\xi ,z,w)}^2
\end{align*}

Let $v_1$, $v_2$ and $v_3$ be obtained from $s_1$, $s_2$, $s_3$ by applying this result. We want to pass to pointwise estimates for the $v_1$, $v_2$ and $v_3$.

The fact that $v_1$, $v_2$ and $v_3$ are ``part'' of holomorphic functions allows us to use subaveraging to get pointwise estimates.

Let $Q(q)=Q({\xi},z,w)$ be the ``largest polydisc'' with center $q\in\partial\Omega\setminus\{p\}$ and $Q(q)\subseteq\Omega_{p}^{**}$. Then
\begin{align*}
\sqrt{\operatorname{Vol}(Q(q))}\geq c\cdot (|u|+\norm{(\xi ,z,w)}^{2L})(|v|+\norm{(\xi ,z,w)}^{2L})|{\Phi}(q)|\text{,}
\end{align*}
where $c>0$. Furthermore we have
\begin{align*}
|v_j (q)|^2 & \leq \widetilde{A} \frac{1}{\operatorname{Vol}(Q)} \int_{Q} |v_j |^2\\
& \leq \widetilde{\widetilde{A}} \frac{\exp({\psi_j}(q))}{\operatorname{Vol}(Q)} \int_{Q} |v_j |^2 e^{-\psi_j ({\eta})}.
\end{align*}
This follows since $\psi_j (q)-\psi_j ({\xi})$ is bounded in $Q$. Now
\begin{align*}
\int_{Q} |v_j |^2 e^{-\psi_j ({\eta})}\leq \int_{\Omega_p^{**}} |v_j |^2 e^{-\psi_j}\leq B<\infty .
\end{align*}
From this we get that
\begin{align*}
|v_j (q)|\leq C \frac{e^{\frac{1}{2}\psi_j (q)}}{\sqrt{\operatorname{Vol}(Q(q))}}
\end{align*}
From the choice of $\psi_0$ and the fact that $d(q,\partial\Omega_p^* )\sim |\Phi (q)|$ when $q\in\partial\Omega\setminus\{p\}$ we obtain the following pointwise estimates:
\begin{align*}
|v_1(q)| & \leq C'\frac{|w|}{|\Phi |^{1+1/(2k)+\delta}}\frac{1}{M_1 M_2}\cdot \norm{(\xi, z,w)}\\
|v_2(q)| & \leq C'\frac{|z|}{|\Phi |^{1+1/(2k)+\delta}}\frac{1}{M_1 M_2}\cdot \norm{(\xi, z,w)}\\
|v_3(q)| & \leq C'\frac{\left( |{\xi}|+|z|^{2L}+|w|^{2L}\right)}{|\Phi |^{1+1/(2k)+\delta}}\frac{1}{M_1 M_2}\cdot \norm{(\xi, z,w)}\text{.}
\end{align*}
Here we have $\norm{(\xi,z,w)}\leq C'' |\Phi|^{\frac{1}{2L}}$ and $\left( |{\xi}|+|z|^{2L}+|w|^{2L}\right)\leq |\Phi|$. The Lemma follows. 
\end{proof}

\section{Proof of the Main Theorem} \label{section:proof}

Now we use the above construction to first create a local kernel, then change coordinates such that each local kernel is given in the same coordinates. We then glue the pieces together and obtain a solution operators on a slightly smaller domains. These solution operators will give uniform estimates which allow us to use a normal families argument to get the solution operator $S_\Omega(f)$ such that $\db S_\Omega(f) =f$ and $\norm{S_{\Omega} (f)}_\infty \leq C(\Omega) \norm{f}_\infty$. We proceed as follows.

\subsection{} We let 
\begin{equation*}
    \Psi(\eta, z) = \sum_{j=1}^{3}{h_j(\eta^0,z)}(\eta_j-z_j),
\end{equation*}
then we see that $\Psi$ is continuous in $\eta$, and hence there exists a neighborhood $U$ of $\eta^0$ in $\di \Omega$ such that $|\Psi(\eta,z)|\geq \frac{1}{2}$ when $\eta \in U =U(\eta^0)$.

From this we get new local solutions to the Cauchy-Fantappie equation by taking $$\widetilde{h}_j=\frac{h_j}{\Psi}.$$

\subsection{} \label{sub:two}

Now choose $\e>0$. Next we use that the estimates we obtain in terms of $\eta^0-z$ translate to similar estimates in terms of $\eta-z$ if $z\in \Omega_{-\e}$ and $\eta \in U$. Note we need to shrink $U$ depending on $\e$.

\subsection{} 

Finally we use a partition of unity relative to a covering of $\di \Omega$ by the family $\{U(\eta^0) \}_{\eta^0\in \di \Omega}$ to glue the solutions to the Cauchy-Fantappie equation together. 

When estimating the kernel in $\C^3$, we need to study terms of the following forms:
\begin{equation} \label{eq:terms}
\frac{\overline{\eta}_i-\overline{z}_i}{\norm{\eta-z}^2}\left(h_j \frac{\di h_k}{\di \zeta} -h_k \frac{\di h_j}{\di \zeta} \right)
\end{equation}
and 
\begin{equation}
h_j \frac{1}{\norm{\eta-z}^3} 
\end{equation} where
\begin{equation}
    \sum_{j=1}^3{h_j(\eta,z)(\eta_j-z_j)}\equiv 1
\end{equation} and $\zeta$ is a complex-tangential variable. 

We start by looking at the situation when $p=\eta^0=(\eta_1^0.\eta_2^0,\eta_3^0)$ is one fixed point in $\di \Omega$ and 
$$(\eta_1^0-z_1,\eta_2^0-z_2, \eta_3^0-z_3)=(\xi, z, w)$$ are the coordinates above. 

Now let

\begin{align*}
h_1 & = g_1 -(\eta_2^0 -z_2) v_1 - (\eta_3^0-z_3)v_2\\
h_2 & = g_2 -(\eta_2^1 -z_1) v_1 - (\eta_3^0-z_3)v_3\\
h_3 & = g_3 -(\eta_1^0 -z_1) v_2 - (\eta_2^0-z_2)v_3.
\end{align*}

When we insert this information into \ref{eq:terms}, we end up with a long list of terms to consider. Representative of the challenges, we need to consider the following:
\begin{align}
& \frac{\overline{\eta}_3^0-\overline{z}_3}{\norm{\eta-z}^2}g_1(\eta_3^0-z_3) \frac{\di v_3}{\di \zeta}\\
& v_1 (\eta_1^0-z_1) (\eta_2^0-z_2) \frac{\di v_1}{\di \zeta} \frac{\overline{\eta}_3^0-\overline{z}_3}{\norm{\eta-z}^2} \\
& g_3(\eta_3^0-z_3) \frac{\di v_2}{\di \zeta} \frac{\overline{\eta}_2^0-\overline{z}_2}{\norm{\eta-z}^2}
\end{align} where $\zeta$ is either the second or third variable. Note that $\frac{\eta_1^0-z_1}{\Phi}$ and $g_3(\eta_3^0-z_3)$ are bounded, and we use Lemma \ref{lem:estimates} to see that all the terms satisfy the following estimates:
\begin{align*}
\left| \frac{\overline{\eta}_3^0-\overline{z}_3}{\norm{\eta-z}^2}g_1(\eta_3^0-z_3) \frac{\di v_3}{\di \zeta}\right| & \leq C \frac{1}{|\Phi|^{1+\frac{1}{2k}+\delta}} |\Phi|^{\frac{1}{2L}} \frac{|\eta_3^0-z_3|^2}{\norm{\eta-z}^2} \frac{1}{M_2(\eta^0,z)M_3(\eta^0,z)} \frac{1}{|\zeta|}\\
\left| v_1 (\eta_1^0-z_1) (\eta_2^0-z_2) \frac{\di v_1}{\di \zeta} \frac{\overline{\eta}_3^0-\overline{z}_3}{\norm{\eta-z}^2} \right| & \leq C \frac{1}{M_2(\eta^0,z)M_3(\eta^0,z)} \frac{1}{|\Phi|^{1+\frac{1}{2k}+\delta}}|\Phi|^{\frac{1}{2L}}\frac{|\eta_2^0-z_2||\eta_3^0-z_3|^2}{\norm{\eta-z}^2}\frac{1}{|\zeta|}\\
\left| g_3(\eta_3^0-z_3) \frac{\di v_2}{\di \zeta} \frac{\overline{\eta}_2^0-\overline{z}_2}{\norm{\eta-z}^2} \right| &\leq C \frac{1}{M_2(\eta^0,z)M_3(\eta^0,z)}\frac{1}{|\Phi|^{1+\frac{1}{2k}+\delta}}|\Phi|^{\frac{1}{2L}} \frac{|\eta^0_2-z_2|^2}{\norm{\eta-z}^2} \frac{1}{|\zeta|}
\end{align*}

If we now use the trick from section \ref{sub:two}, the fact that $\frac{\eta_1-z_1}{\Phi}$ and $\frac{\eta_2-z_2}{\Phi^{\frac{1}{k}}}$ are bounded, and several integration by parts, we see that the integral of the kernel is uniformly bounded independent of $\e>0$.

Finally, as in \cite{R1:holder}, these estimates are now used used to prove the H{\"o}lder estimates in the main theorem.

\section{Some comments about dimensions higher than $3$}\label{sectionhigherdim}
As before, we start with a fixed boundary point $p=(\eta_1 , \dots , \eta_n )$ and smooth solutions $g_1 , \dots , g_n$ to the equation
\begin{align*}
    \sum_{j=1}^{n} g_j(p,z)(\eta_j -z_j)\equiv 1\text{.}
\end{align*}
Then, as before, we obtain
\begin{align}\label{10eyns}
    \sum_{j=1}^{n} \overline{\partial}_z g_j(p,z)(\eta_j - z_j)=0\text{.}
\end{align}
Further:
\begin{align*}
    \overline{\partial}g_i = \overline{\partial}g_i \sum_{j=1}^n g_j (\eta_j -z_j) \text{, }i=1,\dots ,n\text{.}
\end{align*}
From \ref{10eyns} we see that we can replace $\overline{\partial}g_i (\eta_i - z_i )$ by $-\sum_{j\neq i} \overline{\partial}g_j (\eta_j -z_j)$. Iteratively, we follow a Koszul complex-like procedure \cite{MR226387} and get $\overline{\partial}$-closed $(0,q)$-forms until we end up with a $(0,n-1)$-form
\begin{align*}
    \omega = \sum_{j=1}^n \delta_j g_j \overline{\partial}g_1 \wedge \dots \wedge \overline{\partial}g_{j-1} \wedge \overline{\partial} g_{j+1} \wedge \dots \wedge \overline{\partial} g_n\text{,}
\end{align*}
such that $\overline{\partial}\omega =0$. Then we use H{\"o}rmander with weights and matrices (Theorem \ref{thm:hor}) to find $u$, such that $\overline{\partial}u=\omega$. Step by step, we go back and solve $\overline{\partial}$-problems for $(0,q)$-forms with different weights and take advantage of the fact that our data will be multiplied by terms of the form $(\eta_i -z_i )$ in each step backwards to get holomorphic solutions
\begin{align*}
    \sum_{j=1}^n h_j(p,z)(\eta_j -z_j)\equiv 1\text{.}
\end{align*}
Again we choose
\begin{align*}
    g_1 =\frac{1}{\Phi}\text{ and }g_j=\frac{P_j}{\Phi}\text{, }j=2,\dots ,n\text{,}
\end{align*}
where
\begin{align*}
    \Phi (p,z)=(\eta_1 -z_1)-A\cdot F(\eta_2 -z_2 ,\dots ,\eta_n -z_n)
\end{align*}
and $F$ needs to be chosen well. Then
\begin{align*}
    \omega = c\frac{\overline{\partial}P_2\wedge\dots\wedge\overline{\partial}P_n}{\Phi^n}\text{.}
\end{align*}

The big remaining (bumping) problem is to show that locally we can find a {\emph{pseudoconvex}} $\Omega_p^*$ and the the function $F$ such that:

\begin{itemize}
    \item $p\in \partial\Omega_p^*$,
    \item $\overline{\Omega}\setminus\{p\}\subseteq \Omega_p^*$,
    \item $\{z\in\mathbb{C}^n\colon\Phi (p,z)=0\}\cap\overline{\Omega_p^*}=\{p\}$,
    \item If $z\in\partial\Omega$, then $\operatorname{dist}(z,\partial{\Omega_p^*})\sim \Phi (p,z)$.
\end{itemize}

\appendix

\section{\\Bumping Lemma}

We briefly sketch the last remaining case for bumping in $\mathbb{C}^3$. 

Let $\gamma=\{ z\in \mathbb{C}^2 : g(z)=1\}$ where $g$ is holomorphic and homogeneous. Let $P_{2k}=P$ be homogeneous and plurisubharmonic and assume that $P\Big|_\gamma$ is harmonic. Thus $P\Big|_{\gamma_f}$ is also harmonic where $\gamma_f =\{ z \in \mathbb{C}^2 : g(z)=f\}, f\in \mathbb{R}$.

 If the Levi form of $P$ is strictly positive in the normal direction of $\cup_f \{\gamma_f \}=S$, then we proceed as follows.Let $s$ be subharmonic and such that $s \circ g$ is of degree $2k$ (taking roots of $g$ if needed). Now near $S$, we have $P- \e (s \circ g)$ is plurisubharmonic if $\e>0$ is small. Now write $P=P-\e(s\circ g)+\e (s \circ g)$ and bump $s \circ g$ as usual.

If the Levi form if $P$ is not strictly positive in the normal direction. We use a carefully adapted version of a classical bumping trick. Let $P'= P + \delta (A | \Im{g}|^2- |g|^2) |z|^{2s}$ and $s$ is chosen so that $|z|^{2s}|g|^2$ is of degree $2k$. Near $S$, $P' <P$ away from $S$ and the Levi form of $P'$ is strictly positive in the normal direction except near singularities of $g$. However, $s \circ g$ will also have a vanishing Levi form here, so we can choose $A$ to work in this case. Now we use the same trick as the last case; namely,
 we take $(P' - \e (s \circ g))+\e (s \circ g)$ where $P' -\e (s \circ g)$ is plurisubharmonic. 
 
 Now if $\gamma$ is isolated, we can patch with the other bumpings away from $S$. If not, we have two possibilities. First, we have a foliation of ``bad" $\gamma$'s, all level surfaces for $g$, and hence $P=s \circ g$, and we're done. Second, if not, we choose $\gamma_1, \dots, \gamma_n$ from the foliation. We bump near each $\gamma_j$ and carefully choose a cutoff function to glue the bumpings together. By carefully chosen, we mean we take the cutoff function $\chi$ so that it is harmonic in the direction of the complex curves in the foliation.   

Finally, we need to consider the case where $g$ is not homogeneous. Here we need to replace $\Im g$ to get a new function that works the same way locally. Now we choose a partition of unity relative to the local charts. If the derivatives fall on the functions in this partition, then these derivatives will be multiples of something that is zero on our surface $S$. 
 
%%%%%%%%%%%%%%%%%%%%%%%%%%%%%%%%%%%%%%%%%%%%%%%%%%%%%%%%%%%%%%%%%%%%%%%%%%%%%

\def\MR#1{\relax\ifhmode\unskip\spacefactor3000 \space\fi%
  \href{http://www.ams.org/mathscinet-getitem?mr=#1}{MR#1}}

\end{document}